\documentclass[a4paper,11pt]{amsart}
\usepackage[left=2.5cm, right=2.5cm,top=2.5cm,bottom=3cm]{geometry}
\usepackage[utf8]{inputenc}
\usepackage[english]{babel}
\usepackage{makecell}
\usepackage{amssymb}
\usepackage{mathtools}
\usepackage{amsfonts}
\usepackage{amsmath}
\usepackage{amsthm}
\usepackage{dsfont}
\usepackage{soul}
\usepackage{url}
\usepackage{alltt}

\usepackage{xcolor}
\usepackage[normalem]{ulem}
\usepackage{comment}
\usepackage{enumitem}
\usepackage{tikz}

\newtheorem{theoremintro}{Theorem}

\newtheorem{theorem}{Theorem}[section]
\newtheorem{lemma}[theorem]{Lemma}
\newtheorem{cor}[theorem]{Corollary}
\newtheorem{prop}[theorem]{Proposition}

\theoremstyle{definition}
\newtheorem{definition}[theorem]{Definition}
\newtheorem{defprop}[theorem]{Definition and Proposition}
\newtheorem{example}[theorem]{Example}

\theoremstyle{remark}
\newtheorem{remark}[theorem]{Remark}

\title{layered mixed Matrices and Reaction Networks}
\date{\today}

\author[Arne Kuhrs]{Arne Kuhrs\textsuperscript{1}}
\thanks{\textsuperscript{1}Max Planck Institute for Mathematics in the Sciences,
04103 Leipzig, Germany. \texttt{arne.kuhrs@mis.mpg.de}}

\author[Máté L. Telek]{Máté L. Telek\textsuperscript{2}}
\thanks{\textsuperscript{2}Budapest University of Technology and Economics,
1111 Budapest, Hungary. \texttt{mtelek@math.bme.hu}}

\author[Nicola Vassena]{Nicola Vassena\textsuperscript{3}}
\thanks{\textsuperscript{3}Leipzig University,
04109 Leipzig, Germany. \texttt{nicola.vassena@uni-leipzig.de}}

\DeclareMathOperator{\supp}{supp}

\DeclareMathOperator{\Jac}{Jac}

\DeclareMathOperator{\CCF}{CCF}
\DeclareMathOperator{\rank}{rank}

\begin{document}

\begin{abstract}
The purpose of this work is twofold.
In the first part, we consider layered mixed matrices introduced by Murota, relate them to existing notions in combinatorial commutative algebra, and investigate the irreducibility of their determinants.
Furthermore, for a layered mixed matrix in combinatorial canonical form, we determine the sparsity structure of its inverse. That is, we characterize which entries of the inverse are nonzero. 

In the second part, we establish for the first time a formal connection between these algebraic results and the theory of buffering structures for reaction networks developed by Mochizuki and Okada. We identify the lattice of buffering structures with the lattice of order ideals of the block poset of the combinatorial canonical form of the associated layered mixed matrix.
This allows us to
characterize the reducibility of the symbolic Jacobian determinant as a polynomial in the reaction-rate derivatives, as well as the nonzero sensitivity responses of species concentrations to reaction-rate perturbations.
\end{abstract}

\maketitle

\section{Introduction}

A matrix $T$ is called \emph{sparse generic} if its entries are either zero or distinct algebraically independent variables $t_{ij}$. Determinantal ideals of sparse generic matrices are well-studied objects in combinatorial commutative algebra, see e.g. \cite[\S 2]{Kretschmer2024} for an overview. Giusti and Merle \cite{GiustiMerle} showed that the codimension, primeness, and Cohen--Macaulayness of such determinantal ideals are controlled by the largest all-zero subrectangle of $T$, while Boocher \cite{Boocher} determined their minimal free resolutions.

When the entries of the matrix are linear forms in the variables, the situation becomes more delicate. Nevertheless, under suitable assumptions on these linear forms, many interesting algebraic properties can still be established. Two well-studied settings are the theories of \emph{$1$-generic matrices}, initiated by Eisenbud \cite{Eisenbud1987Resiliency,Eisenbud1988}, and of \emph{column-graded} matrices, initiated by Conca, De Negri, and Gorla~\cite{CDNG,CDNG2,CDG4,CDG3}.

Motivated by stability and sensitivity analysis of reaction networks, as explained below, we consider in this work augmented sparse generic matrices, that is, 
block matrices of the form $[\,T \mid Q\,]$, as well as \emph{linearly transformed sparse generic matrices} of the form $ST$. Here, $T$ is sparse generic, while $S$ and $Q$ are matrices with entries in a field $\mathbb{K}$. A linearly transformed sparse generic matrix $ST$ is always a column-graded matrix (see Sec.~\ref{Sec:LinTransSparseGeneric} for more details), while in Lemma~\ref{Lemma_SRvs1Generic} we give a combinatorial description of when $ST$ is $1$-generic.

In \cite{murota_book}, 
a matrix of the form
$L = [\,T \mid Q\,]$ is called a \emph{layered mixed matrix}, or \emph{LM-matrix} for short. In his seminal work~\cite{murota_book}, Murota showed that LM-matrices also admit remarkable algebraic and combinatorial properties. If $L$ is a square matrix, the irreducibility of the polynomial $\det(L)$ can be characterized combinatorially~\cite[Theorem 4.5.6]{murota_book}. Moreover, if $L$ is \emph{LM-irreducible} (see Sec.~\ref{sec:CCF}), then its inverse is fully dense, i.e., every entry of $L^{-1}$ is nonzero. We refine Murota's sparsity analysis by showing that, for an LM-matrix in \emph{combinatorial canonical form}, the zero pattern of its inverse can be read off combinatorially.

To state this result, we briefly recall the relevant terminology. The combinatorial canonical form (CCF) of $L = [\,T \mid Q\,]$ is a unique finest block
upper-triangular matrix $\bar L$ obtained from $L$ by permuting rows and columns and changing
the basis of its numeric part $Q$. It comes with a partition of the row/column index set $[n]=M_1\sqcup\dots\sqcup M_p$ into \emph{blocks}, whose diagonal submatrices
$\bar L[M_a,M_a]$ are the LM-irreducible diagonal blocks of $\bar L$ (see
Sec.~\ref{sec:CCF} for precise definitions). 
The
blocks carry a natural partial order $\preceq$, recorded by the \emph{CCF directed graph}
$\mathcal D_{\CCF}(\bar L)$, which is the Hasse diagram of this order. The vertices are $M_1,\dots,M_p$ and there is a directed path from $M_a$ to $M_b$ if and only if $M_a \preceq M_b$.
In these terms, our main theorem below says that the nonzero pattern of $\bar L^{-1}$ is governed entirely by reachability in
$\mathcal D_{\CCF}(\bar L)$.

\begin{theoremintro}
[Fully dense blocks in the inverse of a CCF]
\label{Thm:A}
Let $\bar L \in \mathbb{K}(t_{ij})^{n\times n}$ be a nonsingular LM-matrix in CCF-form with row/column partition $[n] = M_1 \sqcup \dots \sqcup M_p$. 

Then the following holds, for $a,b \in [p]$:
\begin{equation*}
\begin{cases}\bar L^{-1}[M_{a},M_{b}]\text{ is fully dense} &\text{if } M_{a}\preceq M_{b},\\
\bar L^{-1}[M_{a},M_{b}]= \mathbf{0}&\text{otherwise.}
\end{cases}
\end{equation*}
That is, $\bar L^{-1}[M_{a},M_{b}]$ is fully dense if and only if there is a directed path $
M_{a}\leadsto M_{b}$
in the associated CCF directed graph $\mathcal D_{\operatorname{CCF}}(\bar L)$.
\end{theoremintro}

By \cite[Theorem 1]{Eisenbud1987Resiliency}, the ideal generated by the maximal minors of a $1$-generic matrix is prime. However, to the best of our knowledge, no analogous result is known for column-graded matrices. We take a first step in this direction and give a combinatorial characterization for square linearly transformed sparse generic matrices. The proof of Prop.~\ref{Lemma:IntroGaleDual} relies on Gale duality, which gives a connection between linearly transformed sparse generic and LM-matrices.

\begin{prop}
\label{Lemma:IntroGaleDual}
    Let $T \in \mathbb{K}(t_{ij})^{n\times m}$ be a sparse generic matrix and $S \in  \mathbb{K}^{m\times n}$ such that $\rank S = m$ and $m \leq n$. For any choice of a Gale dual matrix $Q\in\mathbb K^{n\times(n-m)}$ of $S$, there exists a constant $\delta \in \mathbb{K} \setminus \{0\}$ such that
    \begin{align*}
    \det (ST) = \delta \det([\,T \mid Q\,]).
    \end{align*}
    Moreover, if $\det(ST)\neq 0$, then it is a homogeneous polynomial of degree $m$; and $\det(ST)$ is irreducible if and only if exactly one diagonal block of the CCF of $[\,T\mid Q\,]$ contains variables $t_{ij}$.
\end{prop}

Prop.~\ref{Lemma:IntroGaleDual} is crucial to relate stability and sensitivity analysis of reaction networks to the theory of LM-matrices.
From an abstract perspective, a reaction network can be viewed as a directed weighted hypergraph, where the vertices represent entities interacting via hyperarrows~\cite{jostmulas:2019}. Such fundamental structures model systems from various domains, including chemical and metabolic networks in (bio)chemistry \cite{Fei19}, prey–predator and trophic networks in ecology \cite{yodzis:1988}, epidemiological networks \cite{VAA24ME}, and economic networks \cite{neumann:1945}. The interpretation of vertices and hyperarrows varies: vertices may represent chemical and biological species, epidemiological classes, or economic goods, while hyperarrows represent reactions, interactions (e.g. trophic or infectious), transitions, or production industries.

We write $\mathcal N=(\mathbf M,\mathbf E)$ for a reaction network with species set
$\mathbf M$ and reaction set $\mathbf E$, and denote species by $X_m$ and reactions
by~$\rho$. One approach to the analysis of reaction networks is to study the system of
ordinary differential equations that describes the time-evolution $\mathbf{x}(t)$ of the
species concentrations (or densities):
\begin{equation}
\dot{\mathbf{x}}=f(\mathbf{x}):=S\mathbf{r}(\mathbf{x}),
\end{equation}
where, informally, $S$ represents the incidence matrix of the network, and $\mathbf{r}(\mathbf{x})$ is the vector of reaction rate functions, each associated with a hyperedge. For an equilibrium $\bar{\mathbf{x}}$, i.e. $f(\bar{\mathbf{x}})=0$, one typically linearizes the system and studies the Jacobian matrix $\Jac_f(\bar{\mathbf{x}})$ to infer its stability properties. In particular, the Jacobian determinant plays a central role in understanding important phenomena such as zero-eigenvalue bifurcations and consequent multistationarity, i.e. the simultaneous presence of multiple equilibria $\bar{\mathbf{x}}_1\neq \bar{\mathbf{x}}_2$ under otherwise identical conditions. 

In general, an explicit parametrization of the bifurcation condition may still be hard to determine, and even harder to interpret structurally (e.g. \ biologically). 
However, if we do not prescribe an explicit functional form of the reaction rate functions, but only their dependency pattern $\partial r_\rho/\partial x_m \in \{0, r_{\rho m}\}$ symbolically, then we may view the Jacobian matrix as a product $\Jac_f(\bar{\mathbf{x}})=SR$ of a numeric matrix $S$ with a sparse generic matrix $R$. In such a context, a symbolic factorization of $\det(SR)$ as a polynomial in the symbols $r_{\rho m}$ may reveal hidden forms of modularity in the underlying reaction network, thereby facilitating the parametrization of the bifurcation condition and, more importantly, providing a clearer interpretation of the mechanisms controlling this phenomenon. 

Via Prop.~\ref{Lemma:IntroGaleDual}, we can then address such a factorization problem by studying the LM-matrix $L=[\,R \mid Q\,]$, associated to $SR$, whose rows are indexed by the
reactions and whose columns are indexed by the species together with a choice of basis vectors of $\ker(S)$.

This route has already been substantially pursued in reaction network theory, and specifically in \emph{Structural Sensitivity Analysis} (SSA), as developed by Fiedler, Mochizuki, Okada, and co-authors \cite{MF15, FM15, OM16, BF18, hirono2021structural, YHOM:24, HOM:25}, where the first part of Prop.~\ref{Lemma:IntroGaleDual} appears in various forms: see \cite[Prop.~2.2]{FM15} and  \cite[Thm.~0 in SI]{HOM:25}. 
The second part about the irreducibility of the symbolic Jacobian determinant follows from Murota's LM-matrix theory.
Remarkably, SSA provides nontrivial sufficient conditions for the factorization of the Jacobian determinant. The argument relies on certain subnetworks called \emph{buffering structures}, introduced in \cite{OM16}. A subnetwork is a pair $(\mathbf M',\mathbf E')$ consisting of a subset of the species and a subset of the reactions, the precise definition of a buffering structure and its properties are postponed to Sec.~\ref{sec:RN}.
In particular, the existence of a buffering structure implies a corresponding factorization of $\det(SR)$. Moreover, SSA shows that perturbations of nonlinearity $r_\rho(\mathbf{x})$ associated to a hyperarrow $\rho$ within a buffering structure $\gamma$ remain localized inside $\gamma$, a principle known as the \emph{law of localization}. The law of localization also extends to zero-eigenvalue bifurcations and multistationarity, showing that the bifurcating behavior itself remains confined within such subnetworks \cite{OTM18}.

For the purpose of this paper, we take these results as a starting point. Buffering structures provide a powerful reaction network formulation of sufficient conditions for the factorization of the symbolic Jacobian determinant and for the localization of steady-state responses.
Whether they are also necessary for the determinant to factor has remained open. Murota's LM-matrix theory allows us to answer this positively:
\begin{prop}\label{prop:factintro}
Assume that the symbolic Jacobian determinant $\det(SR)$ of a reaction network is not zero. Then $\det(SR)$ is a reducible polynomial in $\mathbb R[r_{\rho m}]$ if and only
if the network admits a buffering structure $(\mathbf M',\mathbf E')$ whose species set satisfies
$\emptyset\neq\mathbf M'\subsetneq\mathbf M$.
\end{prop}

We now apply Thm.~\ref{Thm:A} to sensitivity analysis.
Each block of the CCF of the LM-matrix
$[\,R \mid Q\,]$  carries a set of reactions and a set of species which gives rise to block subnetworks $v_1,\dots,v_p$.  The \emph{influence graph} $\mathcal{G}$ of a reaction network is the directed graph with vertices $v_1,\dots,v_p$ obtained from the CCF directed graph by reversing all edges (see Sec.~\ref{sec:targeted} for the precise construction).

\begin{theoremintro}\label{theorem:B}
Assume that the symbolic Jacobian determinant $\det(SR)$ of a reaction network is not zero, and let $\mathcal G$ be its influence graph.     A perturbation of a reaction $\rho$ influences the concentration of a species $X_m$ if and only if there is a directed path from $v_b$ to $v_a$ in $\mathcal G$ where $\rho \in v_b$ and $X_m \in v_a$.
\end{theoremintro}

Without relying on Murota’s theory, closely related constructions have appeared in the sensitivity literature: see the `influence graph' in \cite{BF18} and, unrelatedly, the `influence graph' in \cite{HOM:25}, as well as the `hierarchical graph' in \cite{YHOM:24}. We discuss in more detail the relation with these constructions at the end of Sec.~\ref{sec:targeted}.

Both the influence graph from Thm.~\ref{theorem:B} and the factors of the symbolic Jacobian determinant from Prop.~\ref{prop:factintro} can be computed efficiently using the algorithm of~\cite{murota_algo}, which computes the combinatorial canonical form in polynomial time. We provide an implementation of this algorithm as a \texttt{Julia} package, and showcase its applicability to reaction networks of larger size.

The paper is organized as follows. Sec.~\ref{Sec:2} focuses on layered mixed matrices in general, while Sec.~\ref{sec:RN} focuses on their application to reaction networks. More specifically, Sec.~\ref{sec:CCF} recalls Murota's theory of layered mixed matrices and their combinatorial canonical form. Sec.~\ref{sec:mainlm} presents the main result of that section (Thm.~\ref{Thm:A}) and its proof, characterizing the nonzero structure of the inverse of an LM matrix in CCF. Sec.~\ref{Sec:LinTransSparseGeneric} connects LM-matrix theory to column-graded and 1-generic matrices. Turning to applications, we formally introduce reaction networks in Sec.~\ref{ssec:RN} and their symbolic
Jacobian matrix in Sec.~\ref{ssec:Jacdet}. In Sec.~\ref{sec:bs-lattice} we identify Murota's LM-surplus function with the `influence index' from Structural Sensitivity Analysis, deduce that the buffering structures of a network are precisely the order
ideals of the block poset of its CCF in Sec.~\ref{sec:bs-orderideals}, and derive from this the reducibility characterization of the symbolic Jacobian determinant (Prop.~\ref{prop:factintro}). Sec.~\ref{sec:sensitivity}
and Sec.~\ref{sec:targeted} then draw the consequences for sensitivity analysis and prove Thm.~\ref{theorem:B}. Implementation and computational issues are presented in Sec.~\ref{Sec:Computation}, while Sec.~\ref{sec:discussion} closes the paper with a discussion of our results and future directions.

\color{black}
\medskip
\paragraph{\textbf{Notation. }}
We write $[n]$ for the set $\{1,\dots,n\}$, and $2^{[n]}$ for the set of subsets of $[n]$ and $\binom{[n]}{m}$ for the set of its $m$-element subsets. Given a matrix $L$  with row index set $\mathcal R$ and
column index set $\mathcal C$, and subsets $I\subseteq \mathcal R$ and $J\subseteq \mathcal C$, we
denote by $L[I,J]$ the submatrix of $L$ obtained by selecting the rows indexed by $I$ and the
columns indexed by $J$, and we write $L_{ij}$ for the entry of $L$ in row $i$ and column $j$. 
For a partially ordered set $(P,\preceq)$ and $x\in P$ we write
$\downarrow x:=\{y\in P\mid y\preceq x\}$ for the principal order ideal generated by~$x$.

\section{Layered mixed matrices}
\label{Sec:2}
\subsection{The combinatorial canonical form (CCF)} \label{sec:CCF}
We begin by recalling some basic notions from the theory of layered mixed matrices, developed by Murota \cite{murota_book}. We then discuss their combinatorial canonical form and review their key properties, closely following the presentation in \cite[\S 4]{murota_book}. Although the theory of LM-matrices extends naturally to nonsquare matrices, we restrict our attention to square matrices throughout this paper for the sake of simplicity.

Throughout this paper, we let $\mathbb{K}$ denote a fixed but arbitrary field, and $\mathbb{K}(t_{ij})$ the fraction field of the polynomial ring $\mathbb{K}[t_{ij}]$ in the variables $t_{ij}$ for $i,j \in [n]$.

\begin{definition}[LM-matrix]
A matrix $L \in \mathbb{K}(t_{ij})^{n\times n}$ is called a \emph{layered mixed matrix}
(or \emph{LM-matrix}) if, after a permutation of its columns,
it can be written as a block matrix
\begin{equation*}
L\cong
[\,T \mid Q\,],
\end{equation*}
where, for $n=n_T+n_Q$, $Q \in \mathbb{K}^{n\times n_Q}$ and $T \in  \mathbb{K}(t_{ij})^{n\times n_T} $ such that $T_{ij} = 0$ or $T_{ij} = t_{ij}$ for each  $i \in [n], j \in [n_T]$. We refer to $T$ as the \emph{symbolic part of $L$} and to $Q$ as the \emph{numeric part of $L$}. 
\end{definition}

\begin{remark}[Indexing of the symbolic entries]\label{Rmk:VariableIndexing}
The nonzero entries of the symbolic part of an LM-matrix are pairwise distinct algebraically independent variables. Throughout the statements and proof of this paper we index them by their position in the matrix at hand: if the entry of an LM-matrix
$L$ in row $i$ and column $j$ is symbolic, we write $L_{ij}=t_{ij}$. However, in displayed examples we sometimes may keep instead the names inherited from an LM-equivalent matrix, so that the equivalence remains visible.
\end{remark}

\begin{remark}[Column vs row LM-matrices]
Murota's original formulation (cf. \cite{murota_book}) defines LM-matrices by rows, namely as matrices whose rows can be permuted into the form
\begin{equation*}
    \tilde{L}\cong \begin{pmatrix}
        Q \\ T
    \end{pmatrix},
\end{equation*}
where $Q \in \mathbb{K}^{n_Q\times n}$ and $T \in  \mathbb{K}(t_{ij})^{n_T\times n} $.
We chose instead the column formulation because it aligns with the notation commonly used in the reaction network literature.
The two formulations are nevertheless equivalent, since any column-form LM-matrix $L$ can be studied through its transpose $L^\top=\tilde L$, which is a row-form LM-matrix, and transposition preserves rank, determinant, reducibility, and related structural properties.
\end{remark}

Next, we recall an equivalence relation on LM-matrices due to Murota that preserves both rank and the irreducibility of the determinant over the polynomial ring $\mathbb{K}[t_{ij}]$. Under this equivalence relation, the rows and columns of an LM-matrix may be permuted, and the basis of the column space of its numeric part may be changed.

\begin{definition}
Two LM-matrices $L=[\,T \mid Q\,], \;L' \in \mathbb{K}(t_{ij})^{n\times n}$ are \emph{LM-equivalent} if there exist permutation matrices $P_r, P_c \in \mathrm{GL}_{n}(\mathbb{K})$ and $B\in \mathrm{GL}_{n_Q}(\mathbb{K})$ such that
\begin{align*}L' = 
P_r \cdot 
[T \mid QB] \cdot
P_c.
\end{align*}
\end{definition}

\begin{example}
\label{Ex:RunningStart}
Consider the following LM-matrices, which will serve as our running example:
    \begin{align*}
  & L = [T\mid Q] = \begin{pmatrix}
 t_{11}  & 0 &  0 &  2 &  1&   0\\
t_{21} & t_{22} &  0 &  1 &  0 &  0\\
0 & t_{32} &   0 &  0  & 2&   0\\
0 & t_{42} &   0 &   0 & -1&   0\\
 0  &  0 &t_{53} &   1 &  1 &  1\\ 
 0 &  0& t_{63} &  1 &  1 &  1
  \end{pmatrix}, \qquad  
  L' = \begin{pmatrix}
  t_{11} & 2 & 0  & 1 & 0 & 0\\
 t_{21} &  1  & t_{22}  & 0 & 0 & 0\\
  0 &   0 & t_{32}  & 2 & 0 & 0\\
   0 & 0 &   t_{42}  & -1 &  0 & 0\\
   0 & 0 & 0 & 0 & t_{53} & 1 \\
      0 & 0 & 0 & 0 & t_{63} & 1
  \end{pmatrix}
\end{align*}
  
The matrix $L'$ is obtained from $L$ by first subtracting the last column from the first two columns of $Q$, and then permuting the resulting columns. Hence, $L$ and $L'$ are LM-equivalent.
\end{example}

As we observed in Example~\ref{Ex:RunningStart}, some matrices within an LM-equivalence class might admit an upper block triangular structure. To find these representative matrices systematically, Murota introduced a representative of each equivalence class, called the \emph{combinatorial canonical form} (CCF). The construction of the CCF relies on the following submodular function. Given an LM-matrix $L \in \mathbb{K}(t_{ij})^{n\times n}$, we decompose the set of column indices $[n] = N_T \sqcup N_Q$ such that $L[[n],N_T]$ and  $L[[n],N_Q]$ give the symbolic part and the numeric part of $L$ respectively. 
The \emph{LM-surplus function} $\varphi_L$ is defined as

\begin{equation}
\label{Eq:surplus}
\begin{aligned}
\varphi_L\colon 2^{[n]}&\longrightarrow \mathbb{Z},\\
I&\longmapsto \rank L[I,N_Q]
   +\bigl|\{\,j\in N_T \mid L_{ij}\neq 0 \text{ for some } i\in I \,\}\bigr|-|I| .
\end{aligned}
\end{equation}

The key facts about $\varphi_L$ are the following:
\begin{itemize}
    \item the rank of $L$ is determined by the minimum of $\varphi_L$ \cite[Theorem 4.2.5]{murota_book}, i.e. $$\rank(L) = \min \{\;\varphi_L(I) \; \mid \; I \subseteq [n] \} + n.$$
    \item $\varphi_L$ is submodular \cite[Eq. (4.17)]{murota_book}, i.e.
    \begin{equation*}
        \varphi_L(I_1\cup I_2)+\varphi_L(I_1 \cap I_2) \le \varphi_L(I_1)+\varphi_L(I_2), \quad \text{for all $I_1,I_2\subseteq [n]$}.
    \end{equation*}
    \item $\varphi_L$ is invariant under LM-equivalence, i.e. if $L'$ is LM-equivalent to $L$, then $$\varphi_L(I) = \varphi_{L'}(\pi(I)), \quad \text{ for all } I \subseteq [n],$$
    where $\pi\colon [n] \to [n]$ denotes the permutation sending the row indices of $L$ to the corresponding row indices of $L'$. 
\end{itemize}

Since $\varphi_L$ is submodular, the family of minimizers is
\[
\mathcal{L}_{\min}(\varphi_L):=\{\, I\subseteq [n] \mid \varphi_L(I)\leq \varphi_L(I')\text{ for all }I'\subseteq [n]\,\}
\]
forms a sublattice of
$2^{[n]}$ \cite[Theorem 2.2.5]{murota_book}.
By Birkhoff's representation theorem \cite[Theorem 2.2.10]{murota_book}, the sublattice
$\mathcal L_{\min}(\varphi_L)\subseteq 2^{[n]}$ determines a partition of $[n] = M_1 \sqcup \dots \sqcup M_p$ with a partial order $\preceq$ among $M_1,\dots, M_p$. 
We write $M_{a}\prec M_{b}$ when
$M_{a}\preceq M_{b}$ and $M_{a} \neq M_{b}$, and $M_{a}\prec' M_{b}$ when
$M_{a}\prec M_{b}$ and there is no $c \in [p]$ with $M_a \prec M_c \prec M_b$.

Due to our choice of column-LM matrices, it is more convenient to consider the opposite partial order to the one used by Murota \cite[Eq. (2.24)]{murota_book}. With this convention, the CCF of a (column-)LM matrix becomes upper block triangular.

To keep the exposition as simple as possible, in this paper we assume that $L$ is nonsingular. By~\eqref{Eq:surplus} we then have
\[
\varphi_L(\emptyset)=\varphi_L([n])=0=\min\varphi_L ,
\]
so both $\emptyset$ and $[n]$ are minimizers. Hence
$\max\mathcal L_{\min}(\varphi_L)=[n]$ and $\min\mathcal L_{\min}(\varphi_L)=\emptyset$.

For the reader's convenience, we briefly recall the construction of the partition and the associated partial order used here.
For any maximal descending chain in the lattice $\mathcal{L}_{\min}(\varphi_L)$
\[ 
[n] = I_0 \supsetneq I_1 \supsetneq \dots \supsetneq I_{p-1} \supsetneq I_{p} = \emptyset,\]
we set
\begin{align}
\label{Eq:partition}
M_a \coloneqq I_{a-1} \setminus I_{a} \quad  \text{ for } a \in [p]
\end{align}
which defines a partition $[n]=M_1\sqcup\dots\sqcup M_p$, and we equip the blocks with
\begin{align}\label{Eq:PartialOrder}
 M_{a} \preceq M_{b} \iff \bigl(\text{for all } I \in \mathcal{L}_{\min}(\varphi_L)\colon
 M_{a} \subseteq I \Rightarrow M_{b} \subseteq I \bigr).
\end{align}

For a given nonsingular LM-matrix $L$, we call the partition $[n] = M_1 \sqcup \dots \sqcup M_p$, as defined in~\eqref{Eq:partition}, the \emph{row/column partition} of $L$ and equip the blocks $M_1,\dots, M_p$ with the partial order defined in~\eqref{Eq:PartialOrder}.
We define the \emph{CCF directed graph $\mathcal{D}_{\CCF}(L)$} as the Hasse diagram of that partial order.
The vertices of $\mathcal{D}_{\CCF}(L)$ are the blocks $M_1,\dots, M_p$ and there is a directed edge from $M_a$ to $M_b$ if and only if $M_a \prec' M_b$. Furthermore, we write $M_a \leadsto M_b $ if there exists a directed path from $M_a$ to~$M_b$ in $\mathcal{D}_{\CCF}(L)$.

Note that although~\eqref{Eq:partition} is defined through a chosen maximal descending chain, neither the
partition nor the partial order~\eqref{Eq:PartialOrder} depends on that choice (see \cite[Theorem 2.2.10]{murota_book}). Moreover, the function $\varphi_L$, and hence
$\mathcal L_{\min}(\varphi_L)$, are invariant under LM-equivalence, so the row/column partition,
the partial order $\preceq$ and the CCF directed graph $\mathcal D_{\CCF}(L)$ are canonically
attached to $L$ and are independent of the choice of a combinatorial canonical form of $L$.
\begin{lemma}
\label{Lemma:PartOrderpreserving}
Let $L \in \mathbb{K}(t_{ij})^{n\times n}$ be a nonsingular LM-matrix with row/column partition $[n] = M_1 \sqcup \dots \sqcup M_p$ and let $a,b \in [p]$. If $M_a \preceq M_b$, then $a \leq b$. 
\end{lemma}

Murota observes this consistency of the block indexing with the partial order in
\cite[p.~50]{murota_book}, in the sentence following Eq.~(2.24). Since our
maximal chain is descending and our partial order is the opposite of his, we include a proof.

\begin{proof}
    Let $I_0 \supsetneq \dots \supsetneq I_p$ be a maximal descending chain in the lattice $\mathcal{L}_{\min}(\varphi_L)$. We prove the claim by contraposition. Assuming that $a > b$, we construct a minimizer containing $M_a$ but not $M_b$, which shows $M_a\not\preceq M_b$ by \eqref{Eq:PartialOrder}. Take $I:=I_{a-1}\in\mathcal{L}_{\min}(\varphi_L)$. By \eqref{Eq:partition}, $M_a = I_{a-1}\setminus I_a\subseteq I_{a-1}=I$. Since $a>b$, i.e. $a-1\ge b$, the chain gives $I_{a-1}\subseteq I_b$ and $M_b=I_{b-1}\setminus I_b$ is disjoint from $I_b$, hence from $I_{a-1}$. We get $M_b\not\subseteq I_{a-1}=I$. Thus $M_a\subseteq I$ while $M_b\not\subseteq I$, so $M_a\not\preceq M_b$.
\end{proof}

\begin{theorem}[Combinatorial Canonical Form]\label{thm:ccf}
Let $L \in \mathbb{K}(t_{ij})^{n\times n}$ be a nonsingular LM-matrix. 
Then there exists a partition $[n] = M_1 \sqcup \dots \sqcup M_p$ 
together with an LM-matrix $\bar L \in \mathbb{K}(t_{ij})^{n\times n}$ that is LM-equivalent to $L$ and satisfies:
\begin{enumerate}[label=\textnormal{(CCF\arabic*)}]
    \item The matrix $\bar L$ is upper block-triangular with respect to these partitions, i.e.
    \begin{equation*}
    \bar L[M_a,M_b]=\mathbf{0} \qquad \text{whenever }\quad 1\leq b < a \leq p.
    \end{equation*}

\item The partition $M_{1} \sqcup \dots \sqcup M_p$ agrees with that defined by the lattice $\mathcal{L}_{\min}(\varphi_{ L})$ as in~\eqref{Eq:partition}, and for the partial order $\preceq$ from~\eqref{Eq:PartialOrder} and for each $a,b \in [p]$, we have
\begin{align*}
& \bar{L}\left[M_a, M_b\right]=\mathbf{0} \quad \text{ if } M_a \text{ and } M_b \text{ are not comparable w.r.t } \preceq,\\
& \bar{L}\left[M_a, M_b\right] \neq \mathbf{0} \quad  
 \text { if } \quad M_a \prec' M_b.
\end{align*}

    \item The diagonal blocks have maximal possible rank:
    \begin{equation*}
    \rank  \bar L[M_a,M_a] =|M_a|
    \qquad \text{ for } a \in [p].
    \end{equation*}
    \item The submatrices of the diagonal blocks satisfy
    \begin{equation*}
        \rank \bar L[M_a \setminus \{i\},M_a \setminus \{j\}] = |M_a |-1 \qquad \text{ for each } a \in [p], i,j \in M_a.
    \end{equation*}
    \item $\bar L$ is the finest proper block-triangular matrix that is LM-equivalent to $L$: if
    $\hat L$ is LM-equivalent to $L$ and upper block-triangular with respect to a partition
    $[n]=\hat M_1\sqcup\dots\sqcup\hat M_q$ into nonempty parts whose diagonal blocks have
    maximal rank, that is
    $\rank \hat L[\hat M_k,\hat M_k]=|\hat M_k|$ for $k \in [q]$, then each $\hat M_k$ is a union
    of blocks of the partition $M_1\sqcup\dots\sqcup M_p$.
\end{enumerate}
\end{theorem}

Thm.~\ref{thm:ccf} is a special case of \cite[Theorem 4.4.4]{murota_book} for arbitrary LM-matrices, not necessarily square. The block structure of a combinatorial canonical form is unique, although the entries in the numerical part are not uniquely determined. To illustrate, we revisit our running example.

\begin{example}
\label{Ex:runningLattice}
 Consider the matrix $L$ from Example~\ref{Ex:RunningStart}. A direct computation shows that the lattice of minimizers of the LM-surplus function as in~\eqref{Eq:partition} equals
 {\small
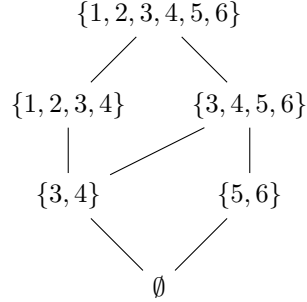
\[
\begin{tikzpicture}[scale=1.2]
  \node (top) at (0,3) {$\{1,2,3,4,5,6\}$};
  \node (left) at (-1,1) {$\{3,4\}$};
  \node (right) at (1,1) {$\{5,6\}$};
   \node (left2) at (-1,2) {$\{1,2,3,4\}$};
  \node (right2) at (1,2) {$\{3,4,5,6\}$};
  \node (bottom) at (0,0) {$\emptyset$};

  \draw (bottom) -- (left) -- (left2) -- (top);
  \draw (bottom) -- (right) -- (right2) -- (top);
    \draw (left) -- (right2);
\end{tikzpicture}
\]}
By choosing a maximal descending chain $I_0 = \{1,2,3,4,5,6\}, I_1 = \{3,4,5,6\}, I_2 =\{5,6\}, I_3 = \emptyset$, we get the partition $M_1 = \{1,2\}, M_2 = \{3,4\}, M_3 = \{5,6\}$. Among these sets only $M_1$ and $M_2$ are comparable with respect to the partial order from~\eqref{Eq:PartialOrder}. We have $M_1 \prec' M_2$ and the matrix $L'$ from Example~\ref{Ex:RunningStart} is a combinatorial canonical form of $L$.

Using column operations on the numerical entries within each block, we get another combinatorial canonical form of $L$. For example, scaling the second or last column of $L'$ by arbitrary nonzero elements of $\mathbb{K}$ yields new combinatorial canonical forms of $L$.
\end{example}

\begin{remark}
The construction of a CCF is based on the LM-surplus function $\varphi_L$. While one can compute a CCF directly from this definition for small examples, this method scales poorly for larger matrices. In Sec.~\ref{Sec:Computation}, we present our \texttt{Julia} implementation of Murota's algorithm~\cite{murota_book, murota_algo}, which computes a CCF of an LM-matrix in polynomial time.
\end{remark}

We conclude this section with two reformulations of theorems from \cite{murota_book} that will be important in our applications to reaction networks. To that end, we need the following terminology.

\begin{definition}\label{Def:lmirr}
 A nonsingular LM-matrix $L \in \mathbb{K}(t_{ij})^{n\times n}$ is called \emph{LM-irreducible} if its \\row/column partition~\eqref{Eq:partition} consists of only one block.  
 
 A nonsingular LM-matrix $L \in \mathbb{K}(t_{ij})^{n\times n}$ is called \emph{essentially LM-irreducible} if its combinatorial canonical form has at most one diagonal block that contains variables $t_{ij}$, and all
the other diagonal blocks are $1\times 1$ matrices over $\mathbb{K}$.
\end{definition}

\begin{theorem}[Irreducibility, {\cite[Theorem~4.5.6]{murota_book}}] \label{thm:murotairr}
Let $L \in \mathbb{K}(t_{ij})^{n\times n}$ be a nonsingular LM-matrix with nonempty symbolic part. Then $\det(L)$ is an irreducible polynomial if and only if $L$ is essentially LM-irreducible.
\end{theorem}

\begin{theorem}[Inverse, {\cite[Theorem 4.5.4]{murota_book}}] 
\label{thm:murotainv}
Let $L \in \mathbb{K}(t_{ij})^{n\times n}$ be a nonsingular LM-matrix. If $L$ is LM-irreducible, then
 $L^{-1}$ is fully dense, i.e., $(L^{-1})_{ij} \neq 0$ for each $i,j \in [n]$. 
\end{theorem}

\subsection{Inverses of matrices in combinatorial canonical form}\label{sec:mainlm}
The following theorem
extends Thm.~\ref{thm:murotainv} to LM-matrices that are not LM-irreducible.

\begin{theorem}[Fully dense blocks in the inverse of a CCF, Theorem~\ref{Thm:A}]\label{thm:ccfinv}
Let $\bar L \in \mathbb{K}(t_{ij})^{n\times n}$ be a nonsingular LM-matrix in CCF with row/column partition $[n] = M_1 \sqcup \dots \sqcup M_p$. 

Then the following holds, for $a,b \in [p]$:
\begin{equation*}
\begin{cases}\bar L^{-1}[M_{a},M_{b}]\text{ is fully dense} &\text{if } M_{a}\preceq M_{b},\\
\bar L^{-1}[M_{a},M_{b}]= \mathbf{0}&\text{otherwise.}
\end{cases}
\end{equation*}
That is, $\bar L^{-1}[M_{a},M_{b}]$ is fully dense if and only if there is a directed path $
M_{a}\leadsto M_{b}$
in the associated CCF directed graph $\mathcal D_{\operatorname{CCF}}(\bar L)$.
\end{theorem}

\begin{remark}[The hypothesis ``in CCF'' is essential]\label{Rmk:CCFnotinvariant}
Thm.~\ref{thm:ccfinv} is a statement about a combinatorial canonical form $\bar L$, and not
about an arbitrary member of its LM-equivalence class. As noted in Sec.~\ref{sec:CCF}, the
row/column partition, the partial order $\preceq$ and the graph $\mathcal D_{\CCF}$ are invariants
of the LM-equivalence class, the zero pattern of the inverse, however, is \emph{not}. Thus for an
LM-matrix $L$ that is not in CCF, the conclusion of Thm.~\ref{thm:ccfinv}
may fail.

To illustrate this, consider the LM-matrix with empty symbolic part
\[
L=\begin{pmatrix}1&1\\0&1\end{pmatrix},
\qquad\text{so that}\qquad
L^{-1}=\begin{pmatrix}1&-1\\0&1\end{pmatrix}.
\]
Here $N_Q=\{1,2\}$ and $\varphi_L\equiv 0$, so every subset of $[2]$ is a minimizer and
$\mathcal L_{\min}(\varphi_L)=2^{[2]}$. The row/column partition is therefore $M_1=\{1\}$,
$M_2=\{2\}$, and the two blocks are incomparable with respect to~\eqref{Eq:PartialOrder}. The
conclusion of Thm.~\ref{thm:ccfinv} would give $L^{-1}[M_1,M_2]=\mathbf 0$, which is false. The
theorem does not apply, because $L$ is not in combinatorial canonical form: $M_1$ and $M_2$ are
incomparable while $L[M_1,M_2]=1\neq 0$, contradicting \textnormal{(CCF2)}. Taking
$B=\bigl(\begin{smallmatrix}1&-1\\0&1\end{smallmatrix}\bigr)$ gives the LM-equivalent matrix
$\bar L=LB=\bigl(\begin{smallmatrix}1&0\\0&1\end{smallmatrix}\bigr)$, which is a combinatorial canonical form of $L$, and indeed
$\bar L^{-1}[M_1,M_2]=\mathbf 0$ as predicted.
\end{remark}

The remainder of this section is devoted to the proof of Thm.~\ref{thm:ccfinv}, which we divide into several preliminary results. We begin by recalling a technique due to Brehm and Fiedler~\cite[Sec. 6]{BF18}, originally introduced to study the `transitivity of influence' in the context of sensitivity analysis for reaction networks, see also Sec.~\ref{sec:sensitivity}.

\begin{lemma}[Transitivity Lemma]\label{lem:transition}
Let $L \in \mathbb{K}(t_{ij})^{n\times n}$ be a nonsingular LM-matrix, and suppose that $L_{ij}=t_{ij}$ for some $i,j \in [n]$ indexed as in Remark~ \ref{Rmk:VariableIndexing}. Then, for any pair of indices $\iota,\eta \in [n]$, we have
\[
-\,\partial_{t_{ij}}\,(L^{-1})_{\iota\eta}=( L^{-1})_{\iota i}\,(L^{-1})_{j\eta}.
\]
In particular, if $( L^{-1})_{\iota i} \neq 0$ and $(L^{-1})_{j\eta} \neq 0$,
then $(L^{-1})_{\iota\eta}\neq 0.$
\end{lemma}

\proof
Applying the product rule for entrywise differentiation to the identity matrix $LL^{-1}$ with respect to the variable $t_{ij}$ yields
\begin{align}
\label{Eq:TransitivityProof}
\mathbf{0} = \partial_{t_{ij}}(L\, L^{-1}) = (\partial_{t_{ij}} L)\, L^{-1} + L\,\left(\partial_{t_{ij}} (L^{-1})\right).
\end{align}
Denote by $e_i$ the $i$-th standard basis vector of size $n \times 1$. Since only the entry $L_{ij} = t_{ij}$ depends on~$t_{ij}$, it follows that $\partial_{t_{ij}} L = e_i e_j^{\top}$.  From \eqref{Eq:TransitivityProof}, by multiplying on the left by $L^{-1}$, we get
\[
-\,\partial_{t_{ij}}\, (L^{-1}) =  L^{-1}(\partial_{t_{ij}} L)\, L^{-1} =  L^{-1}e_i e_j^{\top} L^{-1}.
\]
Evaluating this equality for the $(\iota,\eta)$ entry gives
\begin{equation}\label{eq:transstep}
-\,\partial_{t_{ij}}\,( L^{-1})_{\iota\eta} = (L^{-1})_{\iota i}\,(L^{-1})_{j\eta}.
\end{equation}
If the right-hand side in \eqref{eq:transstep} is nonzero, then $(L^{-1})_{\iota\eta}$ depends on $t_{ij}$, in particular it is nonzero.
\endproof

The Transitivity Lemma~\ref{lem:transition} will be the engine for the proof of Thm.~\ref{thm:ccfinv}. Informally speaking, it allows us to deduce from ``two nonzero entries" that a ``third entry is nonzero" provided there is a variable $t_{ij}$ sitting at the right place.
The following lemma will be crucial for identifying such a variable.

\begin{lemma}[A symbol in every row]\label{lem:symrow}
Let $\bar L \in \mathbb{K}(t_{ij})^{n\times n}$ be a nonsingular LM-matrix in CCF with row/column partition $[n] = M_1 \sqcup \dots \sqcup M_p$ and let $a \in [p]$ with $\vert M_a\vert \ge 2$. Then for every $i \in M_a$ there exists $j \in M_a$ such that $\bar L_{ij} =t_{ij}$.
\end{lemma}
\proof
Consider the diagonal block $\bar{L}' = \bar L[M_a,M_a]$, and denote $N_T, N_Q \subseteq M_a$ the indices of columns giving the symbolic and numeric part of $L'$ respectively.
Since the diagonal block $\bar{L}'$ is LM-irreducible, from \cite[Thm 4.5.1]{murota_book} it follows that $\mathcal{L}_{\min}(\varphi_{L'}) = \{\emptyset, M_a\}$.

Suppose the $i$-th row of $\bar{L}'$ did not contain any variable. In particular, we have that
\[|\{\, j\in N_T \mid L'_{ij} = t_{ij}\,\}| = 0,\]
Since $\bar{L}'$ has full rank, $\bar{L}'$ cannot contain any zero row, which implies that $\rank(\bar{L}'[\{i\},N_Q])= 1$. Thus, $\varphi_{L'} ( \{i\} ) = 1 + 0 -1 = 0$ and $\{i\} \in  \mathcal{L}_{\min}(\varphi_{L'}) = \{ \emptyset, M_a\}$, which contradicts $\vert M_a \vert \geq 2$. 
\endproof

The next ingredient is the  inversion formula for an upper block triangular matrix.

\begin{lemma}[Inverse of a block upper triangular matrix]
\label{lemma:blockinverse}
Let $L \in \mathbb{K}(t_{ij})^{n\times n}$ be a nonsingular upper block triangular matrix
with $p$ square diagonal blocks $X_{11},\dots,X_{pp}$,
\begin{align*}
    L = \begin{pmatrix}
        X_{11} & X_{12} & \dots & X_{1p}   \\
        0 & X_{22}  & \dots  & X_{2p} \\
         \vdots &  &\ddots&\vdots\\
             0 & 0  & \dots & X_{pp} 
    \end{pmatrix}, \qquad    L^{-1} = \begin{pmatrix}
        Y_{11} & Y_{12}  & \dots  &Y_{1p} \\
        0 & Y_{22}   & \dots  & Y_{2p} \\
         \vdots &  &\ddots&\vdots\\
             0 & 0  & \dots &  Y_{pp} 
    \end{pmatrix}.
\end{align*}
Then every diagonal block $X_{aa}$ is invertible, and the blocks of $L^{-1}$ are given
by
\begin{align*}
   (i)\quad & Y_{aa} = X_{aa}^{-1} \ \text{ for all } a\in[p], 
   \qquad Y_{ab} = -X_{aa}^{-1} \sum_{k=a+1}^{b} X_{ak}\, Y_{kb} \ \text{ for } a<b,\\
  (ii)\quad & Y_{ab}
=\sum_{a=a_0<a_1<\dots<a_s=b}(-1)^{s}\,
X_{a_0a_0}^{-1}X_{a_0a_1}X_{a_1a_1}^{-1}X_{a_1a_2}\cdots X_{a_{s-1}a_s}X_{a_sa_s}^{-1},
\end{align*}
where in \textnormal{(ii)} the sum runs over all strictly increasing chains from $a$ to
$b$ in $[p]$.
\end{lemma}

\begin{proof}
Formula (i) is the classical inversion formula for a nonsingular block upper triangular
matrix. We prove (ii) for $a<b$ by induction on the gap
$b-a\ge 1$.

\emph{Base case: $b-a=1$.} The only chain from $a$ to $b=a+1$ is $a<a+1$ (length $s=1$).
By~(i),
\[
Y_{a,a+1}=-X_{aa}^{-1}X_{a,a+1}Y_{a+1,a+1}=-X_{aa}^{-1}X_{a,a+1}X_{a+1,a+1}^{-1},
\]
which is exactly the corresponding term of (ii).

\emph{Inductive step: $b-a\ge 2$.} Assume~(ii) holds for all pairs with a strictly smaller gap. By~(i),
\[
Y_{ab}=-X_{aa}^{-1}\sum_{k=a+1}^{b} X_{ak}Y_{kb}
=\underbrace{-X_{aa}^{-1}X_{ab}X_{bb}^{-1}}_{k=b}
\;-\;X_{aa}^{-1}\!\!\sum_{k=a+1}^{b-1} X_{ak}Y_{kb}.
\]
The first term is the length-one chain $a<b$. For $a+1\le k\le b-1$ the gap $b-k$ satisfies
$1\le b-k<b-a$, so by the inductive hypothesis $Y_{kb}$ is given by~(ii). 
Thus, we have
\begin{align*}
    Y_{ab} = -X_{aa}^{-1}X_{ab}X_{bb}^{-1}
\;+\;\!\!\sum_{k=a+1}^{b-1} \sum_{k=k_0<k_1<\dots<k_s=b}(-1)^{s+1}\, X_{aa}^{-1}X_{ak} \,X_{k_0k_0}^{-1}X_{k_0k_1}\cdots X_{k_sk_s}^{-1}
\end{align*}
As $k$ ranges over
$\{a+1,\dots,b-1\}$ and the inner sums range over all chains from $k$ to $b$, these terms
run exactly once over all chains from $a$ to $b$ of length $\ge 2$. Together with the
length-one chain $a<b$, this is precisely the sum in~(ii). 
\end{proof}

With these preliminary lemmas in place, we are now ready to prove Thm.~\ref{thm:ccfinv}. We begin by considering the case of zero blocks and blocks satisfying $M_a \prec' M_b$.

\begin{prop}
\label{Prop:InverseCCFformula}
  Let $\bar L \in \mathbb{K}(t_{ij})^{n\times n}$ be a nonsingular LM-matrix in CCF with row/column partition $[n] = M_1 \sqcup \dots \sqcup M_p$, and let $a, b \in [p]$. 
    \begin{itemize}
    \item[(i)] If $a > b$, then $\bar L^{-1}[M_{a},M_{b}]=\mathbf{0}$.  
        \item[(ii)] If $a \leq b$ and $M_{a},\, M_{b},$ are not comparable w.r.t. $\preceq$, then  $\bar L^{-1}[M_{a},M_{b}]=\mathbf{0}$. 
        \item[(iii)]  If  $M_{a}\prec' M_{b}$, then $\bar L^{-1}[M_{a},M_{b}] =- \bar L[M_{a},M_{a}]^{-1} \cdot \bar L[M_{a},M_{b}] \cdot \bar L[M_{b},M_{b}]^{-1}$.
        \end{itemize}
\end{prop}

\begin{proof}
Since $\bar L$ is an upper block triangular matrix, its inverse is also block triangular by Lemma~\ref{lemma:blockinverse}. This shows part (i).

If $a \leq b$, then by Lemma~\ref{lemma:blockinverse}, we also have that $\bar{L}^{-1}[M_{a},M_{b}] $ equals
   {\small
    \begin{align}
    \label{Eq:ProofInverseFormula}
\sum_{a=a_0<a_1<\dots<a_s=b}(-1)^{s}\,
\bar L[M_{a_0},M_{a_0}]^{-1} \bar L[M_{a_0},M_{a_1}] \bar L[M_{a_1},M_{a_1}]^{-1}\cdots  \bar L[M_{a_{s-1}},M_{a_s}]\bar L[M_{a_s},M_{a_s}]^{-1}.
    \end{align}
    }
The proof of both parts (ii) and (iii) relies on the following observation. Let  $a=a_0<a_1<\dots<a_s=b$ be a strictly increasing chain such that  
   \begin{align}
    \label{Eq:ProofInverseFormulaSumman}
       \bar L[M_{a_0},M_{a_0}]^{-1} \bar L[M_{a_0},M_{a_1}] \bar L[M_{a_1},M_{a_1}]^{-1}\cdots  \bar L[M_{a_{s-1}},M_{a_s}]\bar L[M_{a_s},M_{a_s}]^{-1} \neq 0.
   \end{align}
   All the blocks $\bar L[M_{a_0},M_{a_1}], \bar L[M_{a_1},M_{a_2}], \dots ,\bar L[M_{a_{s-1}},M_{a_s}]$ are non-zero, which by (CCF2) in Thm.~\ref{thm:ccf} and Lemma \ref{Lemma:PartOrderpreserving} implies that $M_a = M_{a_0} \preceq M_{a_1} \preceq  M_{a_2} \preceq \dots \preceq M_{a_{s-1}} \preceq M_{a_s} = M_b$.

   Now, we prove part (ii) by contraposition. Assume that $\bar L^{-1}[M_{a},M_{b}]\neq \mathbf{0}$. Thus, from~\eqref{Eq:ProofInverseFormula} it follows that there must exist a strictly increasing chain $a=a_0<a_1<\dots<a_s=b$ such that the expression in~\eqref{Eq:ProofInverseFormulaSumman} is nonzero. Thus, the above observation implies that  $M_a \preceq  M_b$.

If  $M_{a}\prec' M_{b}$, then $a < b$ by Lemma~\ref{Lemma:PartOrderpreserving}.
Note that $-\bar L[M_{a},M_{a}]^{-1} \cdot \bar L[M_{a},M_{b}] \cdot \bar L[M_{b},M_{b}]^{-1}$ is a summand of \eqref{Eq:ProofInverseFormula} for the chain $a < b$ of length $s = 1$. Assume there exists another nonzero summand in \eqref{Eq:ProofInverseFormula} for some strictly increasing chain $a=a_0<a_1<\dots<a_s=b$. By the above observation we have that $M_a = M_{a_0} \preceq M_{a_1} \preceq  M_{a_2} \preceq \dots \preceq M_{a_{s-1}} \preceq M_{a_s} = M_b$. Since $M_a \prec' M_b$, it follows that $s=1$ concluding the proof.
\end{proof}

We prove that the blocks satisfying $M_a \preceq M_b$ are fully dense by induction on the shortest directed path between $M_a$ and $M_b$ in the CCF directed graph $\mathcal{D}_{\CCF}(\bar L)$. The base case, where the path has length zero, follows directly from Thm.~\ref{thm:murotainv}. The case of paths of length one requires a more delicate argument. The induction step for arbitrary path length then follows by applying the Transitivity Lemma~\ref{lem:transition}.

\begin{prop}
\label{Prop:DenseCCFBlockCovering}
  Let $\bar L \in \mathbb{K}(t_{ij})^{n\times n}$ be a nonsingular LM-matrix in CCF with row/column partition $[n] = M_1 \sqcup \dots \sqcup M_p$
  , and let $a, b \in [p]$. If  $M_{a}\prec' M_{b}$ 
  , then $\bar L^{-1}[M_{a},M_{b}]$
  is fully dense.
\end{prop}

\begin{proof}
First, we assume that the off-diagonal block $\bar L[M_a,M_b]$ contains a variable $t_{ij}$ for some $i \in M_a, j \in M_b$. The Transitivity
Lemma~\ref{lem:transition} with $t_{ij}$ yields for any
$\iota\in M_{a},\eta\in M_{b}$
\[
-\partial_{t_{ij}}(\bar L^{-1})_{\iota\eta} =(\bar L^{-1})_{\iota i}\,(\bar L^{-1})_{j\eta}.
\]
Since $\iota, i \in M_{a}$ and $j, \eta \in M_{b}$, both factors on the right hand side are nonzero by Thm.~\ref{thm:murotainv}, which implies that $(\bar L^{-1})_{\iota \eta} \neq 0$ for any $\iota\in M_{a},\eta\in M_{b}$.

In the rest  of the proof, we assume that $\bar L[M_a,M_b]$ does not contain any variable $t_{ij}$. Note that if, for some $j_*\in M_b$, the $j_*$-th column of $\bar L$ contains a symbolic variable, then it is a symbolic
column, so all its nonzero entries are variables. As $\bar L[M_a,M_b]$ contains no variable by assumption, $\bar L[M_a,\{j_*\}]$ is the zero vector.

Let $Q_a \subseteq M_a, T_a \subseteq M_a$ be the set of column indices corresponding to the numeric, resp. symbolic columns of $\bar L$.
We distinguish between two subcases. First, assume that 
\begin{align}
\label{Eq:ProofInverseColspan}
\operatorname{colspan} (\bar L[M_a,M_b])\subseteq\operatorname{colspan}(\bar L[M_a,Q_a]).
\end{align}
Thus, using column operations we eliminate the nonzero entries of $\bar L[M_a,M_b]$ using $\bar L[M_a,Q_a]$. The resulting matrix $\bar{L}'$ is a combinatorial canonical form of $\bar L$ with  $\bar{L}'[M_a,M_b] = 0$, which contradicts the condition (CCF2) in Thm.~\ref{thm:ccf}.

Next, we investigate the case when \eqref{Eq:ProofInverseColspan} does not hold.
We claim that there exist $j_* \in T_a, \eta_* \in M_b$ such that $(\bar L^{-1})_{j_* \eta_*} \neq 0$. Note that
\begin{align}
\label{Eq:ProofInverseColspan2}
 \operatorname{span}\{\bar L[M_a,M_a]e_i:i\in Q_a\}=\operatorname{colspan}(\bar L[M_a,Q_a]).
\end{align}
If $\bar L[M_a,M_a]^{-1}\bar L[M_a,M_b]$ were supported only on the
rows indexed by $Q_a$, then each of its columns 
would be contained in $\operatorname{span}\{e_i:i\in Q_a\}$. Multiplying by $\bar L[M_a,M_a]$ from the left, we get that the columns of
$$\bar L[M_a,M_a] \bar L[M_a,M_a]^{-1}\bar L[M_a,M_b] = \bar L[M_a,M_b]$$
are contained in~\eqref{Eq:ProofInverseColspan2}. Thus,~\eqref{Eq:ProofInverseColspan} holds, contradicting our assumption.
Therefore \[\bar L^{-1}[M_a,M_a] \bar L[M_a,M_b]\] has a nonzero entry in some row $j_*\in M_a \setminus Q_a = T_a$. By Prop.~\ref{Prop:InverseCCFformula}, we have
$$\bar L^{-1}[M_a,M_b]=(\bar L^{-1}[M_a,M_a]  \bar L[M_a,M_b])(-\bar L^{-1}[M_b,M_b]).$$ 
Since $-\bar L^{-1}[M_b,M_b]$ is invertible, row $j_*$ of $\bar L^{-1}[M_a,M_b]$ is also nonzero:
there exists an entry $(\bar L^{-1})_{j_*\eta_*}\neq 0$ with $j_*\in T_a$ and $\eta_* \in M_b$, proving our claim. 

In the next part of the proof, we show that the entire $j_*$ row of $\bar L^{-1}[M_a,M_b]$  is fully dense. If $|M_{b}|=1$, $\bar L^{-1}_{j_*\eta_*}\neq 0$ by our previous argument, and there is nothing to prove.
If $|M_{b}|\ge2$, then Lemma~\ref{lem:symrow} yields an entry $t_{\eta_* h}$ in the  $\eta_*$-row of $\bar L[M_b,M_b]$, for at least one column $h \in M_b$. Applying the Transitivity Lemma~\ref{lem:transition} to $\bar L$ at $t_{\eta_* h}$, for every $\eta\in M_{b}$,
\[
-\partial_{t_{\eta_* h}} (\bar L^{-1})_{j_*\eta}=(\bar L^{-1})_{j_*\eta_*}\,(\bar L^{-1})_{h\eta}
\neq0.
\]
The first factor is nonzero by construction, and the second factor is nonzero since it is contained in a diagonal block of $\bar L^{-1}$, note that $h,\eta \in M_{b}$.
Thus the $j_*$-row of $\bar L^{-1}[M_a,M_b]$ is fully dense. 

In the last part of the proof, we show that the whole block $\bar L^{-1}[M_a,M_b]$ is fully-dense. Since $j_*\in T_a$, there exists $i_* \in M_a$ such that $\bar L_{i_*j_*} = t_{i_*j_*}$, otherwise $\bar L[M_a,M_a]$ would have a zero column, contradicting (CCF3).
Applying the Transitivity Lemma~\ref{lem:transition}, for every
$\iota\in M_{a},\eta\in M_{b}$, yields
\[
-\partial_{t_{i_*j_*}} (\bar  L^{-1})_{\iota\eta}=(\bar L^{-1})_{\iota i_*}\,(\bar L^{-1})_{j_*\eta}
\neq0.
\]
Since $\iota, i_*\in M_{a}$, the first factor is nonzero. The second factor is nonzero, since it is contained in the $j_*$-th row of $\bar L^{-1}[M_a,M_b]$. We conclude that 
$(\bar L^{-1})_{\iota\eta}\neq 0$ for all $(\iota,\eta)$, hence $\bar L^{-1}[M_a,M_b]$ is fully dense.
\end{proof}

\color{black}

We are now ready for the induction step in the proof of Thm.~\ref{thm:ccfinv}.
\proof[Proof of Thm.~\ref{thm:ccfinv}]
By Prop.~\ref{Prop:InverseCCFformula}, we have that $\bar L^{-1}[M_{a},M_{b}]=\mathbf{0}$ if $a > b$, or $M_a$ and $M_b$ are not comparable with respect to $\preceq$. 
In the rest of the proof we assume that $M_a \preceq M_b$ and prove by induction on the length of the shortest directed path $M_{a}\leadsto M_{b}$ in the associated CCF directed graph $\mathcal{D}_{\CCF}(\bar L)$ that $\bar L^{-1}[M_{a},M_{b}]$ is fully dense. 

For a path of length zero ($a=b$), the diagonal block 
$\bar L[M_{a},M_{a}]$ is LM-irreducible. Thus, $\bar L^{-1}[M_{a},M_{a}] = \bar L[M_{a},M_{a}]^{-1}$ is fully-dense by Thm.~\ref{thm:murotainv}. The case when the shortest path has length one follows from Prop.~\ref{Prop:DenseCCFBlockCovering}.

 Assume that $\bar L^{-1}[M_a,M_b]$ is fully dense for each $a,b \in [p]$ if there exists a directed path between $M_a$ and $M_b$ in the CCF directed graph $\mathcal{D}_{\CCF}(\bar{L})$ of length at most $k-1$ for some $k \geq 2$. 
 Let $a,b \in [p]$ such that between $M_a$ and $M_b$ the shortest directed path $M_a = M_{a_0} \prec' M_{a_1} \prec' \dots \prec' M_{a_{k-1}} \prec' M_{a_k} = M_b$ has length $k$. 

For each $\ell,\ell' \in \{0,\dots,k-1\}, \,\ell \leq \ell'$ or $\ell,\ell' \in \{1,\dots,k\}, \,\ell \leq \ell'$, the shortest path between $M_{a_\ell}$ and $M_{a_{\ell'}}$ has length at most $k-1$, thus these blocks $\bar L^{-1}[M_{a_\ell},M_{a_{\ell'}}]$ are fully dense by the induction hypothesis.
We show that the block $\bar L^{-1}[M_{a_0},M_{a_{k}}]$ is also fully dense by distinguishing two cases similar to the proof of Prop.~\ref{Prop:DenseCCFBlockCovering}.

If $\bar L[M_{a_{k-1}},M_{a_{k-1}}]$ carries a symbol $t_{ij}$ with $i,j \in M_{a_{k-1}}$, then applying the Transitivity Lemma~\ref{lem:transition} to $\bar L$ with differentiation with respect to $t_{ij}$ yields, for any $\iota\in M_{a_0}$, $\eta\in M_{a_k}$
\[
-\partial_{t_{ij}} (\bar L^{-1})_{\iota\eta}= (\bar L^{-1})_{\iota i}\,(\bar L^{-1})_{j\eta},
\]
where $(\bar L^{-1})_{\iota i}$ is an entry of the dense block $\bar L^{-1}[M_{a_0},M_{a_{k-1}}]$ and
$(\bar L^{-1})_{j \eta}$ an entry of the dense block $\bar L^{-1}[M_{a_{k-1}},M_{a_k}]$. Since these both are
nonzero by induction hypothesis, we have that the block $\bar L^{-1}[M_{a_0},M_{a_{k}}]$ is fully dense. 
Here, it is important that the case of shortest paths of length one has already been handled separately in Prop.~\ref{Prop:DenseCCFBlockCovering}. This guarantees that 
$\bar L^{-1}[M_{a_{k-1}},M_{a_k}]$ is fully dense.

If $ \bar L[M_{a_{k-1}},M_{a_{k-1}}]$ does not carry a variable $t_{ij}$, then by Lemma~\ref{lem:symrow} $\vert M_{a_{k-1}} \vert = 1$.
Thus, $\bar L[M_{a_{k-1}},M_{a_k}]$ is a single row. Moreover, $\bar L[M_{a_{k-1}},M_{a_{k-1}}]$ is an invertible element in $\mathbb{K}$.  Since $M_{a_{k-1}} \prec' M_{a_k}$, Thm.~\ref{thm:ccf} implies that $\bar L[M_{a_{k-1}},M_{a_k}] \neq \mathbf{0}$. 
If all the nonzero entries of $\bar L[M_{a_{k-1}},M_{a_k}]$ are elements of $\mathbb{K}$, then by using column operations on these columns and that $\bar L[M_{a_{k-1}},M_{a_{k-1}}] \neq 0$ we can achieve a new CCF of $\bar L$, where  $\bar L[M_{a_{k-1}},M_{a_k}] = \mathbf{0}$, which is a contradiction.
Thus, the block $\bar L[M_{a_{k-1}},M_{a_k}]$ contains a symbolic entry $t_{ij}, \; i \in M_{a_{k-1}}, j \in M_{a_{k}}$. 
Applying the Transitivity Lemma~\ref{lem:transition} to $\bar L$ at $t_{ij}$ yields for $\iota \in M_{a_0}, \eta \in M_{a_{k}}$
\begin{align}
\label{Eq:ProofInverseDenseBlocks2}
-\partial_{t_{ij}}(\bar L^{-1})_{\iota\eta} =(\bar L^{-1})_{\iota i}\,(\bar L)^{-1}_{j \eta}.
\end{align}
Since $\iota \in M_{a_0}$, $i \in M_{a_{k-1}}$, and $j,\eta \in M_{a_{k}}$, by the inductive hypothesis the factors on the right-hand side in \eqref{Eq:ProofInverseDenseBlocks2} are nonzero. This completes the proof. 
\endproof

\begin{example}
We illustrate Thm.~\ref{thm:ccfinv} on our running example.
Consider the matrix $L'$ from Example~\ref{Ex:RunningStart}. By Example~\ref{Ex:runningLattice}, $L'$ is in CCF with row/column partition $M_1 =\{1,2\}, M_2 =\{3,4\}, M_3 =\{5,6\}$ and the only relation between the blocks is $M_1 \prec' M_2$.
Thm.~\ref{thm:ccfinv} predicts that exactly the four blocks indexed by the pairs
$M_1\preceq M_1$, $M_1\preceq M_2$, $M_2\preceq M_2$ and $M_3\preceq M_3$ are nonzero, and that
each of them is fully dense. A direct computation confirms this:
The inverse of $L'$ is of the form
\begin{align*}
    L'^{-1} = \tfrac{1}{\det(L')} \begin{pmatrix}
        \tilde{L}[M_1,M_1] &  \tilde{L}[M_1,M_2] & \mathbf{0} \\
        \mathbf{0} &   \tilde{L}[M_2,M_2]  & \mathbf{0}  \\
        \mathbf{0} & \mathbf{0} &   \tilde{L}[M_3,M_3] 
    \end{pmatrix},
\end{align*}
where the nonzero blocks are given as
\footnotesize
   \begin{align*}
       &\tilde{L}[M_1,M_1] =   \begin{pmatrix}
           -t_{32}t_{53} + t_{32}t_{63} - 2t_{42}t_{53} + 2t_{42}t_{63}  &                      2t_{32}t_{53} - 2t_{32}t_{63} + 4t_{42}t_{53} - 4t_{42}t_{63} \\[0.5em]
t_{21}t_{32}t_{53} - t_{21}t_{32}t_{63} + 2t_{21}t_{42}t_{53} - 2t_{21}t_{42}t_{63}  &  -t_{11}t_{32}t_{53} + t_{11}t_{32}t_{63} - 2t_{11}t_{42}t_{53} + 2t_{11}t_{42}t_{63}
       \end{pmatrix}, \\[0.7em]
            &\tilde{L}[M_1,M_2] =    \begin{pmatrix}
                -2t_{22}t_{53} + 2t_{22}t_{63} + t_{42}t_{53} - t_{42}t_{63}            &              -4t_{22}t_{53} + 4t_{22}t_{63} - t_{32}t_{53} + t_{32}t_{63} \\[0.5em]
t_{11}t_{22}t_{53} - t_{11}t_{22}t_{63} - t_{21}t_{42}t_{53} + t_{21}t_{42}t_{63}  & 2t_{11}t_{22}t_{53} - 2t_{11}t_{22}t_{63} + t_{21}t_{32}t_{53} - t_{21}t_{32}t_{63}
            \end{pmatrix},\\[0.7em]
              & \tilde{L}[M_2,M_2] =     \begin{pmatrix}
-t_{11}t_{53}+t_{11}t_{63}+2t_{21}t_{53}-2t_{21}t_{63}
&
-2t_{11}t_{53}+2t_{11}t_{63}+4t_{21}t_{53}-4t_{21}t_{63}
\\[0.5em]
-t_{11}t_{42}t_{53}+t_{11}t_{42}t_{63}
+2t_{21}t_{42}t_{53}-2t_{21}t_{42}t_{63}
&
t_{11}t_{32}t_{53}-t_{11}t_{32}t_{63}
-2t_{21}t_{32}t_{53}+2t_{21}t_{32}t_{63}
\end{pmatrix},\\[0.7em]
              &   \tilde{L}[M_3,M_3] =  \begin{pmatrix}
-t_{11}t_{32}-2t_{11}t_{42}+2t_{21}t_{32}+4t_{21}t_{42}
&
t_{11}t_{32}+2t_{11}t_{42}-2t_{21}t_{32}-4t_{21}t_{42}
\\[0.5em]
t_{11}t_{32}t_{63}+2t_{11}t_{42}t_{63}
-2t_{21}t_{32}t_{63}-4t_{21}t_{42}t_{63}
&
-t_{11}t_{32}t_{53}-2t_{11}t_{42}t_{53}
+2t_{21}t_{32}t_{53}+4t_{21}t_{42}t_{53}
\end{pmatrix}.
   \end{align*}
\end{example}

\begin{example}
\label{Ex:chain}
To illustrate the induction step in the proof of Thm.~\ref{thm:ccfinv}, consider the
following LM-matrix in CCF: 
\[
\bar L=\begin{pmatrix}
t_{11} & 2 & 0 & 1 & 0 & 0\\
t_{21} & 1 & t_{23} & 0 & 0 & 0\\
0 & 0 & t_{33} & 2 & t_{35} & 0\\
0 & 0 & t_{43} & -1 & 0 & 0\\
0 & 0 & 0 & 0 & t_{55} & 1\\
0 & 0 & 0 & 0 & t_{65} & 1
\end{pmatrix}.
\]
A direct computation shows that the lattice of minimizers of $\varphi_{\bar L}$ is the chain
$\emptyset\subsetneq\{5,6\}\subsetneq\{3,4,5,6\}\subsetneq[6]$, which yields the row/column
partition $M_1=\{1,2\},\,M_2=\{3,4\},\,M_3=\{5,6\}$ with the total order
$M_1\prec' M_2\prec' M_3$. In particular $M_1\preceq M_3$, and $\mathcal D_{\CCF}(\bar L)$ is
the directed path $M_1\to M_2\to M_3$.

Although $M_1\preceq M_3$, the two blocks are \emph{not} directly coupled:
$\bar L[M_1,M_3]=\mathbf 0$. Nevertheless, Thm.~\ref{thm:ccfinv} asserts that
$\bar L^{-1}[M_1,M_3]$ is fully dense, and indeed
\[
\bar L^{-1}[M_1,M_3]
=\frac{t_{35}}{(t_{11}-2t_{21})(t_{33}+2t_{43})(t_{55}-t_{65})}
\begin{pmatrix}
-2t_{23}+t_{43} & 2t_{23}-t_{43}\\[2pt]
t_{11}t_{23}-t_{21}t_{43} & -\,t_{11}t_{23}+t_{21}t_{43}
\end{pmatrix},
\]
all of whose entries are nonzero.

This is precisely the phenomenon handled by the induction step ($k=2$): the density of $\bar L^{-1}[M_1,M_3]$ is not inherited from a direct block, but is transported along the path $M_1\to M_2\to M_3$. 
Concretely, differentiating with respect
to the variable $t_{33}$ of the middle diagonal block $\bar L[M_2,M_2]$ and applying the
Transitivity Lemma~\ref{lem:transition} expresses $\partial_{t_{33}} (\bar L^{-1})_{\iota \eta}$ as a product of entries of the already-dense length-one blocks $\bar L^{-1}[M_1,M_2]$ and $\bar L^{-1}[M_2,M_3]$.
\end{example}

\subsection{Linearly transformed sparse generic matrices}
\label{Sec:LinTransSparseGeneric}
In this section, we introduce \emph{linearly transformed sparse generic matrices}. These are matrices of the form $ST$, where $S \in \mathbb{K}^{m \times n}$ and $T \in \mathbb{K}(t_{ij})^{n \times m}$ is a sparse generic matrix. It is easy to see that the $j$-th column of $ST$ contains only variables $t_{ij}$ for $i \in [n]$, since
\[(ST)_{kj} = \sum_{i =1}^{n} S_{ki} T_{ij}, \qquad \text{ for } k,j \in [m].\]
In particular, this shows that $ST$ is column-graded in the sense of \cite[Section 3]{CDNG}.

In what follows, by relating linearly transformed sparse generic matrices to LM-matrices, we give a combinatorial characterization of the irreducibility of $\det(ST)$. In the case of $\mathbb{K} = \mathbb{R}$, this connection is well-known in the reaction network literature, see e.g. \cite[Prop.~2.1 and Lemma~2.2]{FM15},  or \cite[Thm.~0 in the SI]{HOM:25}.

We assume that $m \leq n$ and $S$ has rank $m$. A matrix $Q \in \mathbb{K}^{n \times (n-m)}$ is called a \emph{Gale dual} of~$S$, if $SQ = \mathbf{0}$ and $\rank(Q) = n-m$, or in other words, the columns of $Q$ form a basis of $\ker(S)$.

\begin{prop}
\label{Prop:CSGaledualTrick}
    Let $T \in \mathbb{K}(t_{ij})^{n\times m}$ be a sparse generic matrix and $S \in  \mathbb{K}^{m\times n}$ such that $\rank S = m$ and $m \leq n$. For any choice of a Gale dual matrix $Q\in\mathbb K^{n\times(n-m)}$ of $S$, there exists a constant $\delta \in \mathbb{K} \setminus \{0\}$ such that
    \begin{align}
    \label{Eq:Prop:GaleDualTrick}
    \det (ST) = \delta \det([\,T\mid Q\,]).
    \end{align}
  Moreover, if $\det(ST)\neq 0$, then it is a homogeneous polynomial of degree $m$, which is irreducible if and only if $[\,T\mid Q\,]$ is essentially LM-irreducible.
\end{prop}

\begin{proof} 
Using the Cauchy-Binet formula (see e.g. \cite[Theorem 2.3]{Prasolov1994}, \cite[(12.11)]{JoswigTheobald2013}), we~have 
\begin{align}
\label{Eq:Proof:CauchyBinet}
    \det(ST) = \sum_{J \in \binom{[n]}{m}} \det(S[[m],J]) \det(T[J,[m]]).
\end{align}
The Leibniz formula for $\det(T[J,[m])$ yields that each nonzero monomial in $\det(T[J,[m])$  has degree $m$, from which it follows that $\det(ST)$ itself, if nonzero, is homogeneous of degree $m$.

In the second step of the proof, we replace the minors of $S$ in~\eqref{Eq:Proof:CauchyBinet} by the corresponding minors of its Gale dual $Q$.
For $J=\left\{j_1, \ldots, j_m\right\} \subseteq[n]$ with $j_1<\cdots<j_m$ and $J^c=\left\{k_1, \ldots, k_{n-m}\right\}, k_1<\cdots<k_{n-m}$, we denote by $\operatorname{sgn}\left(\tau_J\right) \in\{-1,1\}$ the signature of the permutation that sends $(1, \ldots, n)$ to $\left(j_1, \ldots, j_m, k_1, \ldots, k_{n-m}\right)$. Since $Q$ is the Gale dual of $S$, by \cite[Theorem 12.16]{JoswigTheobald2013} there exists $\delta \in \mathbb{K} \setminus \{0\}$ such that
\[  \det(S[[m],J])  = \delta \operatorname{sgn}(\tau_J)\det(Q[J^c,[n-m]]).\]
Now, using Laplace expansion \cite[Theorem 2.4.1]{Prasolov1994} we have that 
\begin{align*}
     \det( [T\mid Q]) =  \sum_{J \in \binom{[n]}{m}} \operatorname{sgn}(\tau_J) \det( T[J,[m]]) \det(Q[J^c,[n-m]]) 
\end{align*}
finishing the proof of~\eqref{Eq:Prop:GaleDualTrick}. Since, by~\eqref{Eq:Prop:GaleDualTrick}, the polynomials $\det(ST)$ and $\det([\,T\,\vert\,Q\,])$ are scalar multiples of each other, one is irreducible if and only if the other is. The last assertion then follows from Thm.~\ref{thm:murotairr}.
\end{proof}

\begin{remark}
\label{Remark:1genericVSLM}
Since $S$ and $Q$ are Gale dual, by \cite[Proposition 2.1.9]{Oxley} we have 
    \begin{align*}
       \rank Q[I,[n-m]] = \rank S[[m],I^c] + \vert I \vert -m, \quad \text{for each  } I \subseteq [n]
    \end{align*}
    Thus, the LM-surplus function of $[\,T\,\vert\, Q\,]$ is determined by $S$ and $T$. In particular, whether the matrix $[\,T\,\vert\, Q\,]$ in Prop.~\ref{Prop:CSGaledualTrick} is LM-irreducible is completely determined by the combinatorics of $S$ and $T$.
\end{remark}

\color{black}

We conclude this section by relating linearly transformed sparse generic matrices to $1$-generic matrices from \cite{Eisenbud1987Resiliency}. This exploration is of purely algebraic interest and independent of Sec.~\ref{sec:RN} about reaction networks.

Recall that a matrix $X\in \mathbb{K}(t_{ij})^{m\times p}$, whose entries are linear forms in the variables $t_{ij}$, is $1$-generic if for all $u \in \mathbb{K}^m \setminus \{\mathbf{0}\}, v\in \mathbb{K}^p \setminus \{\mathbf{0}\}$ we have $u^\top X v \neq 0$.
\begin{lemma}
    \label{Lemma_SRvs1Generic}
   Let $T \in \mathbb{K}(t_{ij})^{n\times p}$ be a sparse generic matrix and $S \in \mathbb{K}^{m\times n}$. The matrix $ST$
    is $1$-generic if and only if for all $j \in [p]\colon$
    \begin{align*}
        \rank S\big[\,[m]\,,\,\supp(T[[n],\{j\}])\,\big] = m.
    \end{align*}
 In particular, if $m > n$ then $ST$ is not $1$-generic.
\end{lemma}

\begin{proof}
    By definition $ST$ is $1$-generic if and only if for all $u \in \mathbb{K}^m \setminus \{\mathbf{0}\}$ and $ v\in \mathbb{K}^p \setminus \{\mathbf{0}\}$ we have
    \begin{align*}
        u^\top ST v = \sum_{i =1}^n \sum_{j \in \supp(T[\{i\},[p])]}  (u^\top S)_i v_j t_{ij} \neq 0,
    \end{align*}
    which happens if and only if for each  $u \in \mathbb{K}^m \setminus \{\mathbf{0}\}$ and $v\in \mathbb{K}^p \setminus \{\mathbf{0}\}$ there exists $i \in [n]$ and $ j \in \supp(T[\{i\},[p]]) \cap \supp(v)$ such that $(u^\top S)_i  v_j \neq 0$. Using that $j \in \supp(T[\{i\},[p]])$ is equivalent to $i \in \supp(T[[n],\{j\}])$, we conclude that $ST$ is $1$-generic if and only if for each $u \in \mathbb{K}^m \setminus \{\mathbf{0}\}$ and each $j \in [p]$ we have 
    \[u^\top S[\,[m] \,,\,\supp(T[[n],\{j\}])\,] \neq 0,\]
    which gives the desired result.
\end{proof}

\begin{prop}
\label{Prop:1genericVSLM}
   Let $T \in \mathbb{K}(t_{ij})^{n\times m}$ be a sparse generic matrix and $S \in  \mathbb{K}^{m\times n}$ such that $ST$ is 1-generic and $\det(ST) \neq 0$ (in particular $\rank S = m$ and $m \leq n$). For any choice of a Gale dual matrix $Q\in\mathbb K^{n\times(n-m)}$ of $S$, the matrix $[\,T\,\vert\, Q\,]$ is essentially LM-irreducible.
    If we additionally assume that $T$ has no zero rows and $S$ has no zero columns, then the matrix $[\,T\,\vert\, Q\,]$ is LM-irreducible.
\end{prop}

\begin{proof}
By \cite[Theorem 1]{Eisenbud1987Resiliency}, $\det(ST)$ is irreducible, which implies that $\det([\,T\,\vert\, Q\,])$ is irreducible by Prop.~\ref{Prop:CSGaledualTrick}. Thus, $[\,T\,\vert\, Q\,]$ is essentially LM-irreducible by Thm.~\ref{thm:murotairr}.

In the rest of the proof we assume that $T$ has no zero rows, $S$ has no zero columns, and write $L = [\,T\,\vert \, Q\,]$.
It follows from Remark~\ref{Remark:1genericVSLM} that, for every $I \subseteq [n]$,
      \begin{align}
      \label{Eq:Proof:1genericvsLMirreducible}
        \varphi_L(I) = \rank S[[m],I^c]  -m + \vert \{j \in [m] \mid \text{there exists } i \in I \text{ with } T_{ij} \neq 0 \} \vert .
    \end{align}
    Our goal is to prove that $\varphi_L(I) > 0$ for each $\emptyset \subsetneq I \subsetneq [n]$, which gives that $L$ is LM-irreducible. For such $I$, we set 
    \[J(I) \coloneqq \{ k \in [m] \mid I \cap \supp(T[[n],\{k\}]) \neq \emptyset \}.\]
    In particular, we have $\varphi_L(I) = \rank S[[m],I^c]  -m + \vert J(I) \vert$. We distinguish two cases. If $J(I) \neq [m]$, then for $k \in [m] \setminus J(I)$, we have $\supp(T[[n],\{k\}]) \subseteq I^c$ by the construction of $J(I)$.
    Since $ST$ is $1$-generic, Lemma~\ref{Lemma_SRvs1Generic} implies that
    \[ m = \rank S[[m],\supp(T[[n],\{k\}])] \leq \rank S[[m],I^c]  \leq m.\]
    Using again~\eqref{Eq:Proof:1genericvsLMirreducible}, the assumption that $T$ has no zero rows and $I \neq \emptyset$, we conclude that $\varphi_L(I) > 0$.

    If $J(I) = [m]$, then~\eqref{Eq:Proof:1genericvsLMirreducible}  becomes $\varphi_L(I) = \rank S[[m],I^c]$, which is strictly greater than $0$, since $I^c \neq \emptyset$ and $S$ has no zero columns. 
    Hence $\mathcal{L}_{\min}(\varphi_{L}) = \{\emptyset,[n]\}$, so the row/column partition consists of a single block and $L$ is LM-irreducible.
    \end{proof}
    
We conclude this section with two examples showing that the assumptions in Prop.~\ref{Prop:1genericVSLM} are necessary for LM-irreducibility.

\begin{example}
    (a) Consider the matrices 
    \begin{align*}
        S = \begin{pmatrix}
            1 & 1
        \end{pmatrix}, \qquad T = \begin{pmatrix}
            t_{11} \\ 0
        \end{pmatrix}.
    \end{align*}
    The determinant $\det(ST) = t_{11}$ is irreducible, but the corresponding LM-matrix
     \begin{align*}
       [\,T\, \vert \, Q \, ] = \begin{pmatrix}
            t_{11} & 1 \\ 0 & -1
        \end{pmatrix}
    \end{align*}
    is only essentially LM-irreducible and not LM-irreducible. Thus, the assumption in Prop.~\ref{Prop:1genericVSLM} that $T$ has no zero rows is necessary.
    
      (b) To show that the assumption on $S$ is necessary, consider the matrices 
    \begin{align*}
        S = \begin{pmatrix}
            1 & 0
        \end{pmatrix}, \qquad T = \begin{pmatrix}
            t_{11} \\ t_{21}
        \end{pmatrix},
    \end{align*}
    for which $\det(ST) = t_{11}$ is irreducible, but the corresponding LM-matrix is not LM-irreducible,
     \begin{align*}
       [\,T\, \vert \, Q \, ] = \begin{pmatrix}
            t_{11} & 0 \\ t_{21} & 1
        \end{pmatrix} \cong \begin{pmatrix}
            1 & t_{21} \\ 0 & t_{11}
        \end{pmatrix}.
    \end{align*}
\end{example}
   \color{black}
\section{Reaction Networks} \label{sec:RN}
\subsection{Networks, kinetics and equilibria}\label{ssec:RN} For consistency of notation, we mostly adopt chemical terminology throughout.
A reaction network $\mathcal{N}$ is a pair of sets $\mathcal{N}=(\mathbf{M},\mathbf{E})$, where $\mathbf{M}=\{X_1,...,X_{|\mathbf{M}|}\}$ is the set of \emph{species}, and $\mathbf{E}$ is the set of \emph{reactions}. A reaction $\rho \in \textbf{E}$ is an ordered association between two distinct nonnegative linear combinations of the species:
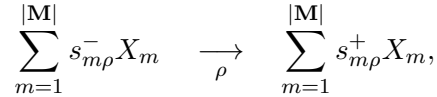
\begin{equation*}\label{eq:reactionj}
   \sum_{m=1}^{|\mathbf{M}|} s^-_{m\rho} X_m\quad \underset{\rho}\longrightarrow \quad \sum_{m=1}^{|\mathbf{M}|} s^+_{m\rho} X_m,
\end{equation*}
where $s^-_{m\rho}, s^+_{m\rho} \in \mathbb{R}_{\ge0}$ are called \emph{stoichiometric coefficients}.  
If $s^-_{m\rho} > 0$, then the species $X_m$ is called an \emph{input} or \emph{reactant} of the reaction $\rho$. If $s^+_{m\rho} > 0$, then $X_m$ is an \emph{output} or \emph{product} of $\rho$.

Throughout the paper, we identify $\mathbf{E}$ with the set $\{1,\dots,\vert \mathbf{E} \vert \}$, $\mathbf{M}$ with the set $\{1,\dots,\vert \mathbf{M} \vert \}$ and use the symbols $\rho$ (resp. $m$) interchangeably to denote both reactions (resp. species) and their indices. With this slight abuse of notation, we define the $|\mathbf{M}|\times |\mathbf{E}|$ \emph{input matrix} $S^- \in \mathbb{R}_{\ge0}^{|\mathbf{M}|\times |\mathbf{E}|}$ and the \emph{output matrix} $S^+ \in \mathbb{R}_{\ge0}^{|\mathbf{M}|\times |\mathbf{E}|}$ by 
\begin{equation*}
    S^-_{m\rho} \coloneqq s^-_{m\rho}\;, \qquad \text{and} \qquad  S^+_{m\rho}\coloneqq s^+_{m\rho}\;.
\end{equation*}
The difference of output and input matrix gives rise to the \emph{stoichiometric matrix} $S$:
\begin{equation}\label{eq:S}
    S:=S^+-S^-\;,
\end{equation}
which encodes the net-consumption and net-production of the reactions.

In reaction network theory, the quantities of interest are classically the concentrations (or densities) of the species. We denote the concentration of $X_m$ by $x_m$, and write $\mathbf{x} \coloneqq (x_1,\dots,x_{\vert \mathbf{M} \vert })$ for the corresponding \emph{concentration vector}.  
Under this interpretation, only real nonnegative values of $\mathbf{x}$ are physically meaningful; therefore, we assume throughout that $\mathbf{x} \in \mathbb{R}_{\geq 0}^{\vert \mathbf{M} \vert }$.

The dynamics of the species concentrations depend on the mathematical laws of the reactions, so-called \emph{kinetics} or \emph{reaction rates}. To emphasize the generality of the approach we are describing, we require here only one basic property of the reaction rates.

\begin{definition}[Kinetic model]\label{def:kineticmodel}
A continuously differentiable function $\mathbf{r}:\mathbb{R}_{\geq  0}^{|\mathbf{M}|}\to \mathbb{R}^{|\mathbf{E}|}_{\geq 0}$ is called a \emph{kinetic model} or a \emph{reaction rate function} for $\mathcal{N}$ if, for every reaction $\rho$, the rate $\mathbf{r}_\rho$ depends exactly on the concentration of the inputs of $\rho$, that is, for each $\rho \in \mathbf{E}$ and $ m\in \mathbf{M}$ we have \begin{equation}\label{eq:dependencies}
        \dfrac{\partial r_\rho(\mathbf{x})}{\partial x_m}\equiv 0 \quad \Leftrightarrow\quad s^-_{m\rho}=0.
    \end{equation}
\end{definition}

\begin{remark}[Further common assumptions on kinetics]
Def.~\ref{def:kineticmodel} only specifies that the dependencies of the reaction rates are encoded in the network structure. We state it in this minimal form since the results presented here do not require further assumptions. To guarantee invariance of the positive orthant under time-evolution, it is customary to additionally assume that $
r_\rho(\mathbf{x})=0$ whenever $x_m=0$ for some input of reaction $\rho$, that is, whenever $s^-_{m\rho}>0$. Further, many results in the literature assume \emph{monotonicity} of the reaction rates, i.e., they prescribe the sign of $\partial r_\rho(\mathbf{x})/\partial x_m$ whenever $s^-_{m\rho}>0$ (and $x_m>0$). We do not require these assumptions here.
\end{remark}

Widely used monotone kinetic models such as classical \cite{MA64} and generalized mass-action kinetics \cite{Muller:12}, Michaelis--Menten kinetics \cite{MM13}, and Hill kinetics \cite{Hill10} satisfy Def.~\ref{def:kineticmodel}, as do nonmonotone kinetics such as Haldane kinetics \cite{haldane:1930} with substrate inhibition. These standard kinetic models all fall within the broad class of  \emph{product-form kinetics}, where
\begin{equation}\label{eq:produtform}
r_\rho(\mathbf{x})=k_\rho\prod_{m=1}^{|\mathbf{M}|} g_{m\rho}(x_m)^{s^-_{m\rho}},
\end{equation}
$k_1,\dots,k_{\vert \mathbf{E} \vert } >0$ are fixed parameters and each $g_{m\rho}$ is a parametric rational function of $x_m$ with $g_{m\rho}(0)=0$. The exponents $s^-_{m\rho}$ ensure that the nonlinearity of $r_\rho$ depends on the concentrations of the reactants of reaction $\rho$. In ecology, mass-action, Michaelis--Menten, and Hill kinetics correspond respectively to Holling functional responses of type I, II, and III \cite{holling:1965}.

After choosing a kinetic model  $\mathbf{r}(\mathbf{x})$, the time evolution of the concentrations, denoted by $\mathbf{x}(t)$, is modelled by the following system of ordinary differential equations (ODEs):
\begin{equation}\label{eq:main}
\dot{\mathbf{x}}=f(\mathbf{x})=S\mathbf{r}(\mathbf{x}),
\end{equation}
where $S$ is the stoichiometric matrix from~\eqref{eq:S}. A nonnegative concentration vector $\bar{\mathbf{x}}\ge0$ is called an \emph{equilibrium} of \eqref{eq:main} if $f(\bar{\mathbf{x}})=0$. For further simplicity, we focus on positive equilibria $\bar{\mathbf{x}}>0$ where all concentrations are positive. A priori, the existence of a positive equilibrium requires a linear condition on $S$, namely that $\ker(S)$ intersects the nonnegative  orthant (resp. positive if $r(\mathbf{x})>0$ for $\mathbf{x}>0$),
which we assume whenever needed. In chemistry, the notion of \emph{steady-state} is used for the analogous concept, as `equilibrium' may carry additional thermodynamic meaning, such as detailed balance.

\begin{example}
\label{ex:RunningNetwork} Consider the following reaction network with $3$ species and $6$ reactions.
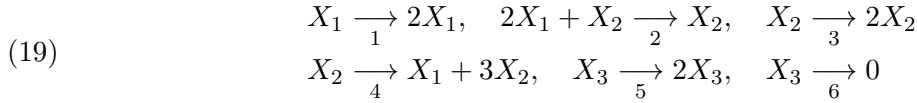
\begin{equation}
\label{Eq:RunningNetwork}
\begin{aligned}
      &X_1    \underset{1}\longrightarrow  2X_1, \quad 
    2X_1  + X_2     \underset{2}\longrightarrow   X_2, \quad 
        X_2     \underset{3}\longrightarrow  2X_2\\ 
        &X_2    \underset{4}\longrightarrow  X_1 +3X_2, \quad
        X_3     \underset{5}\longrightarrow  2X_3, \quad
        X_3     \underset{6}\longrightarrow  0
\end{aligned}
\end{equation}
The network is deliberately abstract and devoid of a biochemical interpretation. Nevertheless, we choose to discuss it here because it displays all the relevant features we wish to highlight, while remaining small enough for the objects involved to be easily understood. For a biologically relevant network, we refer to
Sec.~\ref{Sec:Computation}.
The input matrix and stoichiometric matrix of~\eqref{Eq:RunningNetwork}  equal
\begin{align*}
    S^- = \begin{pmatrix}
    1 &  2 &  0 &  0 &  0 &  0\\
0   &1 &  1 &  1 &  0 &  0\\
0 &  0 &  0 &  0 &  1 &  1
    \end{pmatrix}, \qquad S = \begin{pmatrix}
        1 &  -2 &  0  & 1 &  0 &   0\\
0  &  0 &  1 &  2 &  0 &   0\\
0  &  0 &  0 &  0 &  1 &  -1
    \end{pmatrix}.
\end{align*}
Under the assumption of mass-action kinetics the reaction rate function is given by
\begin{align*}
    \mathbf{r}(\mathbf{x}) = \begin{pmatrix}
        k_1 x_1 &
        k_2 x_1^2x_2 &
        k_3 x_2 &
        k_4 x_2&
        k_5 x_3 &
        k_6 x_3
    \end{pmatrix}^\top,
\end{align*}
where $k_1,\dots,k_6 > 0$ are parameters, called reaction rate constants. One can easily verify that $\mathbf{r}(\mathbf{x})$ satisfies the conditions in Def.~\ref{def:kineticmodel}. For example, in the second reaction, only $X_1$ and $X_2$ are input species, and we have
\begin{align*}
    \tfrac{\partial r_2(\mathbf{x})}{\partial x_1} = 2k_2 x_1x_2, \qquad  \tfrac{\partial r_2(\mathbf{x})}{\partial x_2} = k_2 x_1^2, \qquad  \tfrac{\partial r_2(\mathbf{x})}{\partial x_3} =  0. 
\end{align*}
\end{example}

\subsection{The Jacobian matrix and buffering structures}\label{ssec:Jacdet}

Via Taylor expansion, the linearization of the ODE system \eqref{eq:main} is a main tool in the local analysis of stability and bifurcations of equilibria. Consequently, as a first-order approximation, the Jacobian matrix of $f(\mathbf{x})$ plays a central role. This section focuses on the structure of this Jacobian matrix and relates it to the structural properties of the underlying reaction network.

 Differentiation of the vector field $f(\mathbf{x}) = S \mathbf{r}(\mathbf{x})$ in \eqref{eq:main} gives the Jacobian matrix as: \begin{equation}\label{eq:jacobianformal}
    \Jac_f(\mathbf{x}) = S \Jac_\mathbf{r}(\mathbf{x})\; .
\end{equation}
The non-identically zero entries of $ \Jac_\mathbf{r}(\mathbf{x})$ are specified only by the input-reaction relation, via Eq.~\eqref{eq:dependencies}, namely we have
\begin{equation}\label{eq:equivalent}
\left(\Jac_\mathbf{r}(\mathbf{x})\right)_{\rho m}\not\equiv 0\quad \Leftrightarrow \quad   s^-_{{m\rho}} >0 \quad \Leftrightarrow \quad S^-_{m\rho}>0.
\end{equation}

Motivated by \eqref{eq:equivalent}, we replace each non-identically zero partial derivative $\partial r_\rho(\mathbf{x})/\partial x_m$ by a formal symbol $r_{\rho m}$, algebraically independent from all other such symbols, and define the \emph{symbolic reactivity matrix} $R \in \mathbb{R}(r_{\rho m})^{\vert \mathbf{E}\vert \times \vert \mathbf{M}\vert }$ as
\begin{equation}
\label{Eq:SymbReactivity}
    R_{\rho m} :=
    \begin{cases}
        r_{\rho m}, & \text{if species } X_m \text{ is an input of reaction } \rho
        \text{ (i.e., } s^-_{m\rho} > 0\text{)}, \\
        0, & \text{otherwise}.
    \end{cases}
\end{equation}
Furthermore, we consider the \emph{symbolic Jacobian matrix} 
\begin{equation}\label{eq:symbolicjacobian}
  SR \in \mathbb{R}(r_{\rho m})^{\vert \mathbf{M}\vert \times \vert \mathbf{M}\vert }. 
\end{equation}
Throughout the remainder of the paper, we assume a nondegeneracy condition, namely that the determinant of $SR$ is not identically zero as a polynomial in the variables $r_{\rho m}$:
\begin{equation}\label{eq:nondegcond}
    \det(SR) \neq 0.
\end{equation}

\begin{remark}[Conservation laws]
Via the Cauchy-Binet formula \eqref{Eq:Proof:CauchyBinet}, the nondegeneracy condition \eqref{eq:nondegcond} implies that the stoichiometric matrix $S$ has full rank $\vert \mathbf{M} \vert$. The left-kernel space of~$S$ is thus trivial and 
the reaction network has no nonzero linear conservation laws $w \in \mathbb{R}^{\vert\mathbf{M}\vert}$ s.t.
\[
w\;\dot{\mathbf{x}}=w\;S\mathbf{r}(\mathbf{x})=0.
\]
\end{remark}
\begin{remark}[Qualitative matrix theory]
    By \eqref{eq:equivalent}, the matrix $R$ is a ``symbolized" version of the transpose of the input matrix $S^-$, where we replace the nonzero entries of $(S^{-})^\top$ by algebraically independent variables. If we further assume monotonicity of the reaction rate function $\mathbf{r}$, and $r_{\rho m}>0$, then $R$ reduces to the ``sign-pattern" of $(S^-)^\top$, as it is called in qualitative matrix theory \cite{JefKleeVdD:77}. However, here the Jacobian matrix appears not as a sign-pattern itself, but rather as a product among a real matrix $S=S^+-S^-$ and the sign pattern of $(S^-)^\top$.
\end{remark}
\begin{remark}[Parameter-rich kinetics]\label{rmk:pr}
For the monotone case, the algebraic independence of the partial derivatives $\partial r_\rho(\mathbf{x})/\partial x_m$ can also be achieved by restricting the kinetic model $\mathbf{r}$ to a sufficiently flexible parameterized class. This observation motivated the introduction of the notion of \emph{parameter-rich kinetics} \cite{VasStad24}. Roughly speaking, a parametric kinetic model is said to be parameter-rich if, at any given positive equilibrium $  \mathbf{\bar x}$, any partial derivative $\partial r_\rho(\mathbf{\bar x})/\partial x_m$ can be prescribed independently from each other, subject only to the dependency constraints~\eqref{eq:dependencies}. Within such a framework, for every fixed positive equilibrium $\mathbf{\bar x}$, any symbolic Jacobian matrix~\eqref{eq:symbolicjacobian} can be realized as the evaluated Jacobian matrix $\Jac_f( \mathbf{\bar x})$ for a suitable choice of kinetic parameters. In particular, condition \eqref{eq:nondegcond} is always realizable at an equilibrium. The parameter-rich property holds for several kinetic models commonly used in biochemistry, including Michaelis--Menten, Hill, and generalized mass-action kinetics. In contrast, classical mass-action kinetics is not parameter-rich, and therefore imposes additional algebraic constraints on the Jacobian evaluated at an equilibrium, which require further analysis.
\end{remark}

The symbolic Jacobian matrix~\eqref{eq:symbolicjacobian} is a linearly transformed sparse generic matrix in the sense of Sec.~\ref{Sec:LinTransSparseGeneric}. As in Prop.~\ref{Prop:CSGaledualTrick}, we consider $C \in \mathbb{R}^{\vert \mathbf{E} \vert \times (\vert \mathbf{E} \vert - \vert \mathbf{M} \vert) }$ a Gale dual matrix of $S$ and the associated LM-matrix
\begin{align}
\label{Eq:AMatrix}
A \coloneqq [\,R\mid C\,] \in \mathbb{R}(r_{\rho m})^{\vert \mathbf{E} \vert \times \vert \mathbf{E} \vert}.
\end{align}
By Prop.~\ref{Prop:CSGaledualTrick}, the determinant of the symbolic Jacobian is an irreducible polynomial in $\mathbb{R}[r_{\rho m}]$ if and only if $A$ is essentially LM-irreducible.
The matrix $A$ was originally introduced in~\cite{MF15,FM15} in the context of sensitivity analysis of reaction networks. We return to this perspective in Sec.~\ref{sec:sensitivity}. For now, we focus on the irreducibility of the determinant, and on relating the combinatorial canonical form of $A$ to \emph{buffering structures}, which have been introduced in~\cite{OM16}.

For a reaction network $\mathcal{N} = (\mathbf{M},\mathbf{E})$, a subnetwork
$\mathcal{N}' = (\mathbf{M}',\mathbf{E}')$ consists of a subset of the species
$\mathbf{M}' \subseteq \mathbf{M}$ together with a subset of the reactions
$\mathbf{E}' \subseteq \mathbf{E}$. We call $\mathcal{N}'$ \emph{nontrivial} if it is nonempty (i.e. either $\mathbf{M}'$ or $\mathbf{E}'$ is nonempty) and $\mathcal{N}'\neq\mathcal{N}$. 

Subnetworks are compared and combined componentwise: for two
subnetworks $\gamma_1=(\mathbf M_1',\mathbf E_1')$ and $\gamma_2=(\mathbf M_2',\mathbf E_2')$ we
set
\begin{align}
\label{Eq:SubnetworkOps}
\gamma_1\cup\gamma_2:=(\mathbf M_1'\cup\mathbf M_2',\ \mathbf E_1'\cup\mathbf E_2'),
\qquad
\gamma_1\cap\gamma_2:=(\mathbf M_1'\cap\mathbf M_2',\ \mathbf E_1'\cap\mathbf E_2'),
\end{align}
and write $\gamma_1\subseteq\gamma_2$ if $\mathbf M_1'\subseteq\mathbf M_2'$ and
$\mathbf E_1'\subseteq\mathbf E_2'$. For a set of reactions $\mathbf E'\subseteq \mathbf E$ we write 
\begin{align}
\label{Eq:BSnotation1}
\begin{split}
\kappa(\mathbf E') &:=\dim\{v\in\ker(S)\mid v_\rho = 0 \text{ for all } \rho \in \mathbf{E}\setminus \mathbf{E'} \},\\
\end{split}
\end{align}
The number $\kappa(\mathbf E')$ counts the independent cycles of $\ker(S)$ that are
supported inside $\mathbf E'$. Restriction to the coordinates in $\mathbf E'$ identifies
$\{v\in\ker(S)\mid v_\rho=0 \text{ for all }\rho\in\mathbf E\setminus\mathbf E'\}$ with
$\ker\bigl(S[\mathbf M,\mathbf E']\bigr)$, so the rank--nullity theorem gives the alternative
description
\begin{align}\label{Eq:kappaRank}
\kappa(\mathbf E')=|\mathbf E'|-\rank S[\mathbf M,\mathbf E'].
\end{align}

A subnetwork $\mathcal{N}'$ is called \emph{output-complete} if  $\mathbf{E}'$ contains all reactions $\rho$ where species $X_m\in \mathbf{M}'$ participates as input, that is:
\begin{equation}\label{eq:outputcomplete}
        s^-_{m \rho}>0\quad\text{for $X_m\in \mathbf{M}'$}\quad\Rightarrow \quad \rho \in\mathbf{E}'.
\end{equation} 
Following \cite[Eq.~(11)]{OM16} and \cite[Def.~4]{hirono2021structural}, we define the \emph{influence index} of an output-complete subnetwork $\gamma=(\mathbf M',\mathbf E')$  as
\begin{align}\label{Eq:lambda}
\lambda(\mathbf M',\mathbf E')=|\mathbf E'|-|\mathbf M'|-\kappa(\mathbf E'),
\end{align}
and say that $\gamma$ is a \emph{buffering structure} precisely when $\lambda(\gamma)=0$ \cite[Def.~5]{hirono2021structural}. 

\begin{example} \label{ex:BSrunning}
   Revisit the reaction network from Example~\ref{ex:RunningNetwork}. The corresponding symbolic reactivity matrix and a choice of a Gale dual matrix are given by
   \begin{align*}
       R =  \begin{pmatrix}
 r_{11}  & 0 &  0 \\
r_{21} & r_{22} &  0 \\
0 & r_{32} &   0 \\
0 & r_{42} &   0 \\
 0  &  0 &r_{53} \\ 
 0 &  0& r_{63} 
  \end{pmatrix}, \qquad    C =  \begin{pmatrix}
  2 &  1&   0\\
 1 &  0 &  0\\
 0  & 2&   0\\
  0 & -1&   0\\
    0 &  0 &  1\\ 
 0 &  0 &  1
  \end{pmatrix},
   \end{align*}
   The matrix $A = [ \,  R \, \vert C\,]$ is an LM-matrix whose combinatorial canonical form was computed in Example~\ref{Ex:runningLattice}, we have that
   \begin{align*}
       A \cong \begin{pmatrix}
  r_{11} & 2 & 0  & 1 & 0 & 0\\
 r_{21} &  1  & r_{22}  & 0 & 0 & 0\\
  0 &   0 & r_{32}  & 2 & 0 & 0\\
   0 & 0 &   r_{42}  & -1 &  0 & 0\\
   0 & 0 & 0 & 0 & r_{53} & 1 \\
      0 & 0 & 0 & 0 & r_{63} & 1
  \end{pmatrix}.
   \end{align*}
   Consider the subnetwork ($\mathbf{M}_1' = \{X_1\}$, $\mathbf{E}_1' = \{1,2\}$). Since $X_1$
 serves as an input only to reactions $1$ and $2$, the subnetwork is output-complete. A simple computation shows that
  \begin{align*}
         \lambda(\mathbf M_1',\mathbf E_1') = |\mathbf{E}_1'|  - |\mathbf{M}_1'| - \dim \{ v \in \ker(S)\; | \;  v_3 = v_4 = v_5 = v_6=   0 \} = 2 -1-1 =0.
  \end{align*}
  Thus, $(\mathbf{M}_1',\mathbf{E}_1')$ is a buffering structure. A similar computation shows that $(\mathbf{M}'_2 = \{X_1,X_2\}$, $\mathbf{E}'_2 = \{1,2,3,4\})$ and $(\mathbf{M}'_3 = \{X_1,X_2,X_3\}$, $\mathbf{E}'_3 = \{1,2,3,4,5,6\})$ are also buffering structures.
 In Sec.~\ref{sec:bs-lattice}, we will develop methods to determine all buffering structures of $\mathcal N$, cf. Example~\ref{Ex:SurplusRunning}.
\end{example}

The correspondence between buffering structures and upper triangular blocks of $A$ in Example~\ref{ex:BSrunning} is not a coincidence. The following characterization of buffering structures in terms of upper block triangular matrices suggests a strong connection between 
buffering structures and the CCF of LM-matrices.
We make this connection precise in Sec.~\ref{sec:bs-lattice}~and~\ref{sec:bs-orderideals}. 

\begin{prop}\label{def:bs}
Assume the nondegeneracy condition~\eqref{eq:nondegcond}. The network $\mathcal N$ admits a nontrivial buffering structure if and only if there is a Gale dual matrix $C$ of $S$ such that the rows and the columns of $[\,R\mid C\,]$ can be permuted so that the resulting matrix is
upper block-triangular with respect to a partition of $[\,|\mathbf E|\,]$ into at least two nonempty parts.

Moreover,
for each buffering structure $(\mathbf M',\mathbf E')$ there exist a choice of a Gale dual matrix $C$ of $S$ and a subset of indices $\Gamma \subseteq [ \lvert \mathbf{E} \rvert]$ such that $A = [\,R \mid C\,]$ is upper block triangular (after permuting its rows and columns), that is 
\[
A[\,\Gamma^{c},\Gamma\,]=\mathbf 0 ,
\]
the rows of $\Gamma$ are the reactions of $\mathbf E'$, and the columns of $\Gamma$ are the species columns of $\mathbf M'$ together with $\kappa(\mathbf E')$ columns of $C$.
\end{prop}

\begin{proof}
The result follows directly from the definition of  buffering structures, for more details see
\cite[Eq.~(12)]{OM16} and \cite[Eq.~(4) and SI Appendix~A]{HOM:25}. 
\end{proof}

\subsection{Influence index and the LM-surplus function}
\label{sec:bs-lattice}
As a first step to connect buffering structures to the theory of LM-matrices, we show that the LM-surplus function $\varphi_A$ from
Sec.~\ref{sec:CCF} and the influence index~\eqref{Eq:lambda} are the same
function, evaluated on complementary sets of reactions (Lemma~\ref{Lem:SurplusIndex}). Moreover, as
a consequence we deduce that the buffering structures correspond bijectively to
the minimizers of $\varphi_A$ (Thm.~\ref{Cor:SurplusConsequences}). 

Throughout this subsection, let $\mathcal N=(\mathbf M,\mathbf E)$ be a reaction network satisfying the
nondegeneracy condition~\eqref{eq:nondegcond}, let $C$ be a Gale dual matrix of $S$ and let $A=[\,R\mid C\,]$ be the LM-matrix from~\eqref{Eq:AMatrix}. 
For a set of reactions $\mathbf E'\subseteq \mathbf E$ we write 
\begin{align}
\label{Eq:BSnotation}
\begin{split}
\operatorname{in}(\mathbf E') &:=\{X_m\in\mathbf M\mid s^-_{m\rho}>0 \text{ for some }\rho\in\mathbf E'\},\\
\mathbf M^{\mathrm{oc}}(\mathbf E') &:=\mathbf M\setminus \operatorname{in}(\mathbf E\setminus\mathbf E').
\end{split}
\end{align}
In words: $\operatorname{in}(\mathbf E')$ is the set of species consumed by the reactions in
$\mathbf E'$, and $\mathbf M^{\mathrm{oc}}(\mathbf E')$ consists of those species
that are \emph{not} consumed by any reaction outside $\mathbf E'$.
As the next lemma shows, the set $\mathbf M^{\mathrm{oc}}(\mathbf E')$ of species is exactly what output-completeness allows.

\begin{lemma}\label{Lem:OutputComplete}
A subnetwork $(\mathbf M',\mathbf E')$ is output-complete if and only if $\mathbf M'\subseteq \mathbf M^{\mathrm{oc}}(\mathbf E')$. In particular,
$\bigl(\mathbf M^{\mathrm{oc}}(\mathbf E'),\mathbf E'\bigr)$ is the largest output-complete subnetwork with reaction set $\mathbf E'$.
\end{lemma}

\begin{proof}
By~\eqref{eq:outputcomplete}, $(\mathbf M',\mathbf E')$ is output-complete if and only if 
every reaction that consumes a species of $\mathbf M'$ already belongs to $\mathbf E'$, that is, if and only if no
species of $\mathbf M'$ is consumed by a reaction of $\mathbf E\setminus\mathbf E'$. In the
notation~\eqref{Eq:BSnotation} this reads
$\mathbf M'\cap\operatorname{in}(\mathbf E\setminus\mathbf E')=\emptyset$, which is precisely
$\mathbf M'\subseteq\mathbf M^{\mathrm{oc}}(\mathbf E')$.
\end{proof}

\begin{lemma}[The LM-surplus function is the influence index]\label{Lem:SurplusIndex}
For every $\mathbf E'\subseteq\mathbf E$,
\[
\varphi_A(\mathbf E\setminus \mathbf E')
\;=\;|\mathbf E'|-\bigl|\mathbf M^{\mathrm{oc}}(\mathbf E')\bigr|-\kappa(\mathbf E')
\;=\;\lambda\bigl(\mathbf M^{\mathrm{oc}}(\mathbf E'),\,\mathbf E'\bigr)
\;
\]
\end{lemma}

\begin{proof}
Put $I:=\mathbf E\setminus\mathbf E'$ and recall from Sec.~\ref{sec:CCF} that
\[
\varphi_A(I)=\underbrace{\rank A[I,N_Q]}_{(a)}
+\underbrace{\bigl|\{j\in N_T\mid \exists\, i\in I \text{ with } A_{ij}\neq 0\}\bigr|}_{(b)}
-\underbrace{|I|}_{(c)},
\]
where $N_T$ indexes the columns of $R$, that is the species, and $N_Q$ those of $C$.

\emph{The term $(b)$.} The $j$-th column of $R$ has a nonzero entry in a row indexed by $I$ if and only if the species 
$X_j$ is consumed by some reaction of $I$. 
Hence $(b)=|\operatorname{in}(I)|$.

\emph{The term $(a)$.} The submatrix $A[I,N_Q]$ consists of the rows of $C$ indexed by $I$. Since the columns of $C$ form a basis of $\ker(S)$, this submatrix represents the coordinate projection
\[
\pi_I\colon \ker(S)\longrightarrow \mathbb{R}^{I},\qquad v\longmapsto (v_\rho)_{\rho\in I},
\]
so $(a)=\dim\pi_I(\ker(S))$. A vector $v\in\ker(S)$ lies in $\ker(\pi_I)$ exactly when $v_\rho=0$ for
all $\rho\in I$, that is, when $\supp(v)\subseteq\mathbf E'$; by~\eqref{Eq:BSnotation1} the space of
such $v$ has dimension $\kappa(\mathbf E')$. The rank--nullity theorem therefore gives
\[
(a)= \dim\pi_I(\ker (S)) = \dim\ker(S) - \dim \ker(\pi_I) =
\dim\ker (S)-\kappa(\mathbf E')=\bigl(|\mathbf E|-|\mathbf M|\bigr)-\kappa(\mathbf E'),
\]
where we used that $S$ has full rank $|\mathbf M|$ by~\eqref{eq:nondegcond}.

Using that $(c) = |I|=|\mathbf E|-|\mathbf E'|$, by adding the three terms $(a),(b),(c)$, we obtain
\[
\varphi_A(I)=|\mathbf E'|-\bigl|\mathbf M^{\mathrm{oc}}(\mathbf E')\bigr|-\kappa(\mathbf E'),
\]
which equals $\lambda\bigl(\mathbf M^{\mathrm{oc}}(\mathbf E'),\,\mathbf E'\bigr)$.
\end{proof}

\begin{remark}\label{Rmk:LinkTo1generic}
Lemma~\ref{Lem:SurplusIndex} is the reaction network instance of
Remark~\ref{Remark:1genericVSLM}: there we observed that if $Q$ is a Gale dual of $S$, then the LM-surplus function of $[\,T\mid Q\,]$ is determined by $S$ and $T$ alone. For $T=R$ and $Q=C$
this says that $\varphi_A$ is determined by the stoichiometric matrix $S$ together with the zero
pattern of $R$, that is, by $S$ and $S^-$. Lemma~\ref{Lem:SurplusIndex} makes this dependence
explicit by exhibiting $\varphi_A$ as the influence index. In particular $\varphi_A$, and hence
also the row/column partition of $A$ and the partial order $\preceq$ of~\eqref{Eq:PartialOrder},
do not depend on the choice of the Gale dual matrix $C$.
\end{remark}

Next, we prove the main result of this subsection, establishing a correspondence between buffering structures and the minimizer lattice of the LM-surplus function. We note that part~(i) of Thm.~\ref{Cor:SurplusConsequences} can also be found in \cite[Rem.~8]{hirono2021structural}.

\begin{theorem}
\label{Cor:SurplusConsequences}
As before we assume the nondegeneracy condition~\eqref{eq:nondegcond}. Then:
\begin{enumerate}[label=\textnormal{(\roman*)}]
\item $\lambda(\gamma)\ge 0$ for every output-complete subnetwork $\gamma$,
\item a subnetwork $(\mathbf M',\mathbf E')$ is a buffering structure if and only if
      $\mathbf M'=\mathbf M^{\mathrm{oc}}(\mathbf E')$ and
      $\mathbf E\setminus\mathbf E'\in\mathcal L_{\min}(\varphi_A)$.
\end{enumerate}
In particular, the species set of a buffering structure is determined by its reaction set, and
$(\mathbf M',\mathbf E')\mapsto\mathbf E\setminus\mathbf E'$ is an inclusion-reversing bijection
from the set of buffering structures of $\mathcal N$ onto $\mathcal L_{\min}(\varphi_A)$.
\end{theorem}

\begin{proof}
To prove (i), let $\gamma=(\mathbf M',\mathbf E')$ be an output-complete subnetwork. Using
Lemma~\ref{Lem:SurplusIndex}, we obtain the decomposition
\begin{equation}
\label{Eq:lambdaSplit}
\begin{aligned}
\lambda(\mathbf M',\mathbf E')
&= |\mathbf E'|-|\mathbf M'| -\kappa(\mathbf E') -\bigl|\mathbf M^{\mathrm{oc}}(\mathbf E')\bigr| + \bigl|\mathbf M^{\mathrm{oc}}(\mathbf E')\bigr| =\\
&\lambda( \mathbf M^{\mathrm{oc}}(\mathbf E'),\mathbf E')  +  \Bigl(\bigl|\mathbf M^{\mathrm{oc}}(\mathbf E')\bigr|-|\mathbf M'|\Bigr) 
=\underbrace{\varphi_A(\mathbf E\setminus\mathbf E')}_{\ge\,0}
+\underbrace{\Bigl(\bigl|\mathbf M^{\mathrm{oc}}(\mathbf E')\bigr|-|\mathbf M'|\Bigr)}_{\ge\,0}.
\end{aligned}
\end{equation}
For the first summand, observe that $A$ being nonsingular implies $\min\varphi_A=0$. By Lemma~\ref{Lem:OutputComplete} we have
$\mathbf M'\subseteq\mathbf M^{\mathrm{oc}}(\mathbf E')$, which gives the nonnegativity of the second summand.   

For (ii), suppose first that $\gamma$ is a buffering structure, that is, $\gamma$ is output-complete and $\lambda(\gamma)=0$. Then
both summands in~\eqref{Eq:lambdaSplit} must vanish, which implies that
$\mathbf M'=\mathbf M^{\mathrm{oc}}(\mathbf E')$ and
$\varphi_A(\mathbf E\setminus\mathbf E')=0=\min\varphi_A$, hence
$\mathbf E\setminus\mathbf E'\in\mathcal L_{\min}(\varphi_A)$.
Conversely, if  $\mathbf M'=\mathbf M^{\mathrm{oc}}(\mathbf E')$ and
      $\mathbf E\setminus\mathbf E'\in\mathcal L_{\min}(\varphi_A)$, then $\gamma$ is output-complete by Lemma~\ref{Lem:OutputComplete}, and $\lambda(\gamma)=0$
by~\eqref{Eq:lambdaSplit}. Thus, $\gamma$ is a buffering structure.

Finally, by (ii) a buffering structure is determined by its reaction set, so the map is injective. By the converse direction of (ii), $\bigl(\mathbf M^{\mathrm{oc}}(\mathbf E\setminus I),\,\mathbf E\setminus I\bigr)$ is a buffering structure for every minimizer $I \in \mathcal{L}_{\min}(\varphi_A)$, giving the surjectivity of the map.

It remains to compare the orders. Since $\operatorname{in}(\cdot)$ is monotone, so is $\mathbf E'\mapsto\mathbf M^{\mathrm{oc}}(\mathbf E')$. Hence for buffering
structures the componentwise inclusion $\gamma_1\subseteq\gamma_2$ is equivalent to $\mathbf E_1'\subseteq\mathbf E_2'$ alone, and therefore to
$\mathbf E\setminus\mathbf E_1'\supseteq\mathbf E\setminus\mathbf E_2'$.
\end{proof}

We now illustrate Lemma~\ref{Lem:SurplusIndex} and Thm.~\ref{Cor:SurplusConsequences} on the running example.

\begin{example}\label{Ex:SurplusRunning}
Consider the network of Example~\ref{ex:RunningNetwork} with the Gale dual matrix $C$ as in Example~\ref{ex:BSrunning},
and take $\mathbf E'=\{1,2\}$, so that $I=\mathbf E\setminus\mathbf E'=\{3,4,5,6\}$. The reactions
in $I$ consume $X_2$ and $X_3$, hence $\operatorname{in}(I)=\{X_2,X_3\}$ and
$\mathbf M^{\mathrm{oc}}(\mathbf E')=\{X_1\}$. On the one hand, $I$ is one of the six minimizers of
$\varphi_A$ computed in Example~\ref{Ex:runningLattice}, so $\varphi_A(I)=0$. On the other hand,
the only cycle of $\ker(S)$ supported in $\mathbf E'$ is spanned by $(2,1,0,0,0,0)^{\top}$, so
$\kappa(\mathbf E')=1$ and
\[
\lambda\bigl(\mathbf M^{\mathrm{oc}}(\mathbf E'),\mathbf E'\bigr)=2-1-1=0.
\]
 By Thm.~\ref{Cor:SurplusConsequences}(ii) the pair $(\{X_1\},\{1,2\})$ is
therefore a buffering structure, and it is the only one with reaction set $\{1,2\}$: the
subnetwork $(\emptyset,\{1,2\})$ is output-complete as well, but
$\lambda(\emptyset,\{1,2\})=2-0-1=1\neq 0$, in accordance with~\eqref{Eq:lambdaSplit}.

By contrast, for $\mathbf E'=\{3,4\}$ we have $I=\{1,2,5,6\}$ and $\varphi_A(I)=2>0$. This
reflects the fact that $X_2$ is consumed by reaction~$2$, which lies outside $\mathbf E'$. Hence
$\mathbf M^{\mathrm{oc}}(\mathbf E')=\emptyset$, and indeed $\lambda(\emptyset,\{3,4\})=2-0-0=2$,
since no cycle of $\ker(S)$ is supported in $\{3,4\}$. By
Thm.~\ref{Cor:SurplusConsequences}(ii) no choice of species set makes
$(\,\cdot\,,\{3,4\})$ a buffering structure.

More generally, Thm.~\ref{Cor:SurplusConsequences}(ii) turns the six minimizers of
$\varphi_A$ listed in Example~\ref{Ex:runningLattice} into the six buffering structures of
$\mathcal N$, obtained by complementing the reaction set and taking the largest output-complete
species set:
\[
(\emptyset,\emptyset),\;\;
(\{X_1\},\{1,2\}),\;\;
(\{X_3\},\{5,6\}),\;\;
(\{X_1,X_2\},\{1,2,3,4\}),\;\;
(\{X_1,X_3\},\{1,2,5,6\}),\;\;
(\mathbf M,\mathbf E).
\]

Because the bijection reverses inclusion, the Hasse diagram of these six buffering structures is the diagram of Example~\ref{Ex:runningLattice} turned upside down.
\end{example}

\begin{remark} By Thm.~\ref{Cor:SurplusConsequences}, a reaction set appears in at most one buffering structure. However, a
species set may occur in several buffering structures. For such an example, consider the network
\[
X_1\xrightarrow{\ 1\ }X_2\xrightarrow{\ 2\ }X_3\xrightarrow{\ 3\ }X_1,
\qquad
X_3\xrightarrow{\ 4\ } 0
\]
By computing the minimizer lattice of the corresponding LM-surplus function $\varphi_A$ and using Thm.~\ref{Cor:SurplusConsequences}, one finds six
buffering structures
\[
(\emptyset,\emptyset),\quad
(\{X_1\},\{1\}),\quad
(\{X_2\},\{2\}),\quad
(\{X_1,X_2\},\{1,2\}),\quad
(\{X_1,X_2\},\{1,2,3\}),\quad(\mathbf M,\mathbf E),
\]
whose six reaction sets are precisely the complements of the six minimizers of $\varphi_A$.
The list also shows that the converse of the uniqueness statement in Thm.~\ref{Cor:SurplusConsequences} fails: the species set
$\{X_1,X_2\}$ occurs in two different buffering structures, with reaction sets $\{1,2\}$ and
$\{1,2,3\}$, whereas no reaction set occurs twice.
\end{remark}

\subsection{Buffering structures as order ideals}
\label{sec:bs-orderideals}
Thm.~\ref{Cor:SurplusConsequences} already establishes a connection between buffering structures and LM-matrices. We strengthen this connection by identifying the buffering structures with the order ideals of the CCF block poset (Thm.~\ref{Thm:BSideals}). This identification yields several lattice-theoretic consequences (Cor.~\ref{Cor:BSlattice}) and will be used to characterize the reducibility of the symbolic Jacobian determinant (Prop.~\ref{thm:q}). Along the way, we recover results of
\cite{hirono2021structural} such as the closure of buffering structures under unions and intersections.

Throughout the section let $A$ be the LM-matrix associated to a nondegenerate reaction network as in~\eqref{Eq:AMatrix}. 
Since the rows of $A$ are indexed by the reactions, we consider the row/column partition of $A$ from~\eqref{Eq:partition} as a partition of the set of reactions
\[
\mathbf E=M_1\sqcup\dots\sqcup M_p.
\]
The partition is obtained from a maximal descending chain
$\mathbf E=I_0\supsetneq I_1\supsetneq\dots\supsetneq I_p=\emptyset$ in
$\mathcal L_{\min}(\varphi_A)$ via $M_a=I_{a-1}\setminus I_a$; in particular every $M_a$ is
nonempty. Neither the blocks nor the partial order $\preceq$ of~\eqref{Eq:PartialOrder} depend on
the chosen chain or on the chosen combinatorial canonical form, as noted in
Sec.~\ref{sec:CCF}. 

Fix a combinatorial canonical form $\bar L=P_r\,[\,R\mid CB\,]\,P_c$ of $A$. Since $B$ is
invertible, the columns of $CB$ are again a basis of $\ker(S)$, so replacing $C$ by $CB$ we may
assume from now on that a combinatorial canonical form of $A$ is obtained by permutations alone,
$\bar L=P_r\,A\,P_c$. Let $\tau,\sigma\colon[\,|\mathbf E|\,]\to[\,|\mathbf E|\,]$ be the
permutations induced by $P_r$ and $P_c$, so that the row $\rho$ of $A$ is the row $\tau(\rho)$ of
$\bar L$ and the column $j$ of $A$ is the column $\sigma(j)$ of $\bar L$; here the columns of
$A=[\,R\mid C\,]$ are indexed so that $1,\dots,|\mathbf M|$ are the species.

Since $\varphi_A(I)=\varphi_{\bar L}(\tau(I))$ for all $I\subseteq\mathbf E$, the permutation
$\tau$ carries $\mathcal L_{\min}(\varphi_A)$ onto $\mathcal L_{\min}(\varphi_{\bar L})$, and
hence maximal descending chains to maximal descending chains. By \textnormal{(CCF1)} and
\textnormal{(CCF2)}, the diagonal blocks of $\bar L$, numbered in the order in which they appear
in its block-triangular display, are the row/column partition of $\bar L$ determined
by~\eqref{Eq:partition}, and this numbering is realized by the maximal descending chain in
$\mathcal L_{\min}(\varphi_{\bar L})$ whose $a$-th term is the union of the blocks numbered
$a+1,\dots,p$. We take the preimage of that chain under $\tau$ as the chain used
in~\eqref{Eq:partition}. With this choice $\tau(M_1),\dots,\tau(M_p)$ are the diagonal blocks of
$\bar L$ in display order, so that
\[
\bar L[\tau(M_a),\tau(M_b)]=\mathbf 0\qquad\text{whenever } b<a ,
\]
and moreover $M_a\preceq M_b$ if and only if $\tau(M_a)\preceq\tau(M_b)$.

Each block of $\bar L$ contains, besides the rows of $\tau(M_a)$, a set of species columns. We
put
\begin{align}\label{Eq:BlockSpecies}
\mathbf M_a:=\{\,X_m\in\mathbf M\mid \sigma(m)\in\tau(M_a)\,\},\qquad a\in[p],
\end{align}
for the set of species whose column of $A$ lies in the $a$-th block of $\bar L$. Since
$\tau(M_1),\dots,\tau(M_p)$ partition the column indices of $\bar L$, the sets
$\mathbf M_1,\dots,\mathbf M_p$ partition $\mathbf M$; unlike the $M_a$, they may be empty. The
block $\tau(M_a)$ has $|M_a|$ columns, of which $|\mathbf M_a|$ are species columns and the
remaining $k_a:=|M_a|-|\mathbf M_a|\ge 0$ are columns of $C$; as the columns of $C$ are
distributed among the blocks, $\sum_{a=1}^{p}k_a=|\mathbf E|-|\mathbf M|$. We write
\begin{align}\label{Eq:BlockData}
v_a:=(\mathbf M_a,\,M_a),\qquad a\in[p],
\end{align}
for the subnetwork consisting of the species and the reactions of the $a$-th block. (These will
be the vertices of the influence graph $\mathcal G$ of Sec.~\ref{sec:targeted}.)

The description of the buffering structures now rests on a purely order-theoretic notion, which
we recall from \cite[Rem.~2.2.9]{murota_book}.

\begin{defprop}
\label{defprop:orderideal}
    Let $(P,\preceq)$ be any partially ordered set. A
subset $J\subseteq P$ is an \emph{order ideal}, or a \emph{down-set}, if
\begin{align}\label{Eq:OrderIdeal}
\text{for all } y \in J \text{ and } x\in P:
      \ x\preceq y
   \quad\Longrightarrow\quad
   x\in J
\end{align}
that is, if $J$ is closed under passing to $\preceq$-smaller elements. 

\begin{itemize}
    \item[(i)] The order ideals of
$(P,\preceq)$ form a distributive sublattice of $2^{P}$ with respect to union and intersection.
\item[(ii)] $J$ is an order ideal of $(P,\preceq)$ if and only if $P\setminus J$ is an order ideal of  $P$ equipped with the
opposite order to $\preceq$.
\end{itemize}
\end{defprop}

Our main interest lies in the partially ordered set 
$\mathcal P:=\{M_1,\dots,M_p\}$, equipped with the partial order $\preceq$ of~\eqref{Eq:PartialOrder}.
The next lemma identifies the species consumed by the reactions of a union of blocks, and shows
in particular that the sets $\mathbf M_a$ are canonically attached to the blocks.

\begin{lemma}\label{Lem:BlockSpeciesIdeal}
For any order ideal $J$ of $(\mathcal P,\preceq)$, we have
\begin{align}\label{Eq:ProofInIdeal}
\operatorname{in}\Bigl(\bigcup_{M_a\in \mathcal P\setminus J}M_a \Bigr)=\bigcup_{M_a\in \mathcal P\setminus J}\mathbf M_a .
\end{align}
Consequently, for every $a\in[p]$,
\begin{align}\label{Eq:BlockSpeciesCanonical}
\mathbf M_a=\operatorname{in}\Bigl(\ \bigcup_{M_c\,\succeq\, M_a} M_c\Bigr)\ \setminus\
          \operatorname{in}\Bigl(\ \bigcup_{M_c\,\succ\, M_a} M_c\Bigr).
\end{align}
In particular, the definition of $\mathbf M_a$ does not depend on the choice of the combinatorial canonical form or on the choice of the maximal descending chain used to construct the blocks $M_1,\dots,M_p$.
\end{lemma}

\begin{proof}
We first prove~\eqref{Eq:ProofInIdeal}.
Let $J$ be an order ideal of $(\mathcal P,\preceq)$ and set 
\[U:=\mathcal P\setminus J, \qquad I_J:=\bigcup_{M_a\in U}M_a.\]
Note that $U$ is closed upwards: if $M_a\in U$ and
$M_a\preceq M_b$, then $M_b\in U$, since $M_b\in J$ would force $M_a\in J$
by~\eqref{Eq:OrderIdeal}.

``$\subseteq$'': let $X_m\in\operatorname{in}(I_J)$, say $X_m$ is consumed by a reaction
$\rho\in M_a$ with $M_a\in U$, and let $b$ be the index with $X_m\in\mathbf M_b$. Being consumed
means $A_{\rho m}=r_{\rho m}\neq 0$, that is $\bar L_{\tau(\rho)\,\sigma(m)}\neq 0$; since
$\tau(\rho)\in\tau(M_a)$ and $\sigma(m)\in\tau(M_b)$ by~\eqref{Eq:BlockSpecies}, the block
$\bar L[\tau(M_a),\tau(M_b)]$ is nonzero. By (CCF1) this forces $a\le b$, and by (CCF2) the
blocks $M_a$ and $M_b$ are comparable. If $M_b\preceq M_a$, then $b\le a$ by
Lemma~\ref{Lemma:PartOrderpreserving}, hence $a=b$. Thus, we have  $M_a\preceq M_b$, and
therefore $M_b\in U$, so $X_m\in\bigcup_{M_c\in U}\mathbf M_c$.

``$\supseteq$'': let $M_a\in U$ and $X_m\in\mathbf M_a$, so that $\sigma(m)\in\tau(M_a)$. By
(CCF3) the diagonal block $\bar L[\tau(M_a),\tau(M_a)]$ has rank $|M_a|$ and hence no zero
column, so the column $\sigma(m)$ has a nonzero entry in some row $\tau(\rho)$ with $\rho\in M_a$.
As $\sigma(m)$ is a species column, that entry is $r_{\rho m}$, so $X_m$ is consumed by
$\rho\in M_a\subseteq I_J$.

To prove~\eqref{Eq:BlockSpeciesCanonical}, for a fixed $a \in [p]$ consider the upward closed sets
\[U_1:=\{M_c\in\mathcal P\mid M_a\preceq M_c\}, \quad
\text{ and } \quad U_2 := U_1 \setminus\{M_a\}.\]
By Prop.~\ref{defprop:orderideal}(ii) there exist order ideals $J_1,J_2$ of $(\mathcal{P},\preceq)$ such that $U_1 = \mathcal{P} \setminus J_1, \; U_2 = \mathcal{P} \setminus J_2$.
\bigskip
Applying~\eqref{Eq:ProofInIdeal} to each of them and using that the sets $\mathbf M_c$ are pairwise disjoint, we conclude 
\begin{align*}
    \operatorname{in}\Bigl(\ \bigcup_{M_c \in U_1 } M_c\Bigr)\ \setminus\
          \operatorname{in}\Bigl(\ \bigcup_{M_c \in U_2} M_c\Bigr) = \bigcup_{M_c\in U_1 }\mathbf M_c \setminus \bigcup_{M_c\in U_2}\mathbf M_c =  \mathbf M_a.
\end{align*}
The right-hand side
of~\eqref{Eq:BlockSpeciesCanonical} depends only on $(\mathcal P,\preceq)$, which
by Sec.~\ref{sec:CCF} is canonically attached~to~$A$, finishing the proof.
\end{proof}

Murota's Birkhoff-type description of the minimizers of  $\varphi_A$ \cite[Eq.~(2.27)]{murota_book} states, using
$\min\mathcal L_{\min}(\varphi_A)=\emptyset$ and $\max\mathcal L_{\min}(\varphi_A)=\mathbf E$,
that a subset of $\mathbf E$ is a minimizer of $\varphi_A$ if and only if it is the union of the
blocks indexed by an order ideal of $\mathcal P$ \emph{equipped with Murota's partial order}.
Since the partial order $\preceq$ of~\eqref{Eq:PartialOrder} is the opposite of Murota's, Prop.~\ref{defprop:orderideal}(ii) gives that 
\begin{align}\label{Eq:MinimizersIdeals}
   I\in\mathcal L_{\min}(\varphi_A)
   \quad\Longleftrightarrow\quad
   \mathbf E\setminus I=\bigcup_{M_a\in J} M_a
   \ \text{ for an order ideal } J \text{ of } (\mathcal P,\preceq).
\end{align}
Combining~\eqref{Eq:MinimizersIdeals} with Thm.~\ref{Cor:SurplusConsequences}(ii), the
reaction sets of the buffering structures are exactly the unions
$\bigcup_{M_a\in J}M_a$ over the order ideals $J$ of $(\mathcal P,\preceq)$. The following theorem upgrades
this by including the species sets, and identifying the
two lattices.

\begin{theorem}[Buffering structures are the order ideals of the block poset]\label{Thm:BSideals}
Assume the nondegeneracy condition~\eqref{eq:nondegcond}. For an order ideal $J$ of
$(\mathcal P,\preceq)$ put
\[
\gamma_J:=\Bigl(\ \bigcup_{M_a\in J}\mathbf M_a\ ,\ \bigcup_{M_a\in J}M_a\ \Bigr).
\]
Then $J\mapsto\gamma_J$ is an isomorphism of partially ordered sets from the lattice of order
ideals of $(\mathcal P,\preceq)$, ordered by inclusion, onto the set of buffering structures of
$\mathcal N$, ordered by componentwise inclusion. Its inverse sends a buffering structure
$(\mathbf M',\mathbf E')$ to $\{M_a\in\mathcal P\mid M_a\subseteq\mathbf E'\}$.
\end{theorem}

\begin{proof}
Let $J$ be an order ideal, put $U:=\mathcal P\setminus J$, $\mathbf E_J:=\bigcup_{M_a\in J}M_a$ and
$I_J:=\bigcup_{M_a\in U}M_a$, so that $\mathbf E\setminus\mathbf E_J=I_J$, or equivalently $\mathbf{E} \setminus I_J = \mathbf{E} _J$.
By~\eqref{Eq:MinimizersIdeals} the set $I_J$ is a minimizer of~$\varphi_A$, so
Thm.~\ref{Cor:SurplusConsequences}(ii) yields exactly one buffering structure with
reaction set $\mathbf E_J$, namely $\bigl(\mathbf M^{\mathrm{oc}}(\mathbf E_J),\mathbf E_J\bigr)$.
By Lemma~\ref{Lem:BlockSpeciesIdeal} and because the sets $\mathbf M_a$ partition $\mathbf M$,
\begin{align}
\label{eq:proof:orderideals}
\mathbf M^{\mathrm{oc}}(\mathbf E_J)=\mathbf M\setminus\operatorname{in}(I_J)
=\mathbf M\setminus\bigcup_{M_a\in U}\mathbf M_a=\bigcup_{M_a\in J}\mathbf M_a ,
\end{align}
so that buffering structure is $\gamma_J$.

Conversely, if $(\mathbf M',\mathbf E')$ is a buffering structure, then
$\mathbf E\setminus\mathbf E'\in\mathcal L_{\min}(\varphi_A)$ by
Thm.~\ref{Cor:SurplusConsequences}(ii), and $\mathbf E'=\bigcup_{M_a\in J}M_a$ for an order
ideal $J$ by~\eqref{Eq:MinimizersIdeals}. Using the same argument as in~\eqref{eq:proof:orderideals}, we have that the buffering structure can be written as $(\mathbf M',\mathbf E')=\gamma_J$.
Since the blocks are nonempty and pairwise disjoint, $M_a\subseteq\mathbf E'$ if and
only if $M_a\in J$. Hence, $J$ is recovered from $\gamma_J$ as
$\{M_a\in\mathcal P\mid M_a\subseteq\mathbf E'\}$. Thus, the maps are inverses of each other, and both preserve inclusion.
\end{proof}

\begin{example}\label{Ex:IdealsRunning}
We illustrate Thm.~\ref{Thm:BSideals} on the running example.
Consider the minimizer lattice as in Example~\ref{Ex:runningLattice}
and the maximal descending chain
$\mathbf E\supsetneq\{3,4,5,6\}\supsetneq\{5,6\}\supsetneq\emptyset$. This produces, by
\eqref{Eq:partition} and~\eqref{Eq:BlockSpecies}, the three subnetworks
\[
v_1=\bigl(\{X_1\},\{1,2\}\bigr),\qquad
v_2=\bigl(\{X_2\},\{3,4\}\bigr),\qquad
v_3=\bigl(\{X_3\},\{5,6\}\bigr).
\]
A subset $J\subseteq \mathcal{P}=\{M_1,M_2,M_3\}$ is an order ideal precisely when $M_2\in J$ implies
$M_1\in J$, so exactly two of the eight subsets of $\mathcal{P}$ are excluded, namely $\{M_2\}$ and
$\{M_2,M_3\}$. Thm.~\ref{Thm:BSideals} therefore produces six buffering structures

$\gamma_J=\bigcup_{M_a\in J}v_a$:
\[
\renewcommand{\arraystretch}{1.2}
\begin{array}{c|c|c}
J & \bigcup_{M_a\in J}M_a
  & \gamma_J=\bigl(\bigcup_{M_a\in J}\mathbf M_a,\ \bigcup_{M_a\in J}M_a\bigr)\\\hline
\emptyset & \emptyset & (\emptyset,\emptyset)\\
\{M_1\} & \{1,2\} & (\{X_1\},\{1,2\})\\
\{M_3\} & \{5,6\} & (\{X_3\},\{5,6\})\\
\{M_1,M_2\} & \{1,2,3,4\} & (\{X_1,X_2\},\{1,2,3,4\})\\
\{M_1,M_3\} & \{1,2,5,6\} & (\{X_1,X_3\},\{1,2,5,6\})\\
\mathcal{P} & \mathbf E & (\mathbf M,\mathbf E)
\end{array}
\]

This is exactly the list of six buffering structures found in
Example~\ref{Ex:SurplusRunning}; note that the two extreme ideals give the two trivial buffering
structures $\gamma_\emptyset=(\emptyset,\emptyset)$ and $\gamma_{\mathcal{P}}=(\mathbf M,\mathbf E)$. The
theorem adds to Example~\ref{Ex:SurplusRunning} that the species set of $\gamma_J$ is read off
from $J$ as well, as $\bigcup_{M_a\in J}\mathbf M_a$, rather than having to be recomputed as
$\mathbf M^{\mathrm{oc}}$ of the reaction set.

Finally, since $J\mapsto\gamma_J$ is a lattice
isomorphism, the componentwise operations~\eqref{Eq:SubnetworkOps} can be read off from the
ideals: for instance
\begin{align}
\label{Eq:RunningExIllustratePart1}
(\{X_1,X_2\},\{1,2,3,4\})\cup(\{X_1,X_3\},\{1,2,5,6\})=(\mathbf M,\mathbf E)
\text{ and } \\
(\{X_1,X_2\},\{1,2,3,4\})\cap(\{X_1,X_3\},\{1,2,5,6\})=(\{X_1\},\{1,2\}),
\end{align}
corresponding to $\{M_1,M_2\}\cup\{M_1,M_3\}=\mathcal{P}$ and $\{M_1,M_2\}\cap\{M_1,M_3\}=\{M_1\}$.
\end{example}

To state the next result, we need the following terminology. Abstractly, a partially ordered set $\mathcal{L}$ (e.g., the set of buffering structures ordered by inclusion) is called a \emph{lattice} if it admits two operations, called the \emph{``meet"} $\cap$ and the \emph{``join"} $\cup$, satisfying certain conditions. Since, for our purposes, these operations will always be the usual intersection and union of sets, respectively, we do not go into further details here. For a precise definition, we refer to~\cite[Remark 2.2.14]{murota_book}.
An element $Z \in \mathcal{L}$ is said to be \emph{join-irreducible} if $Z=Z^{\prime} \cup Z^{\prime \prime}$ with $Z^{\prime}, Z^{\prime \prime} \in \mathcal{L}$ means $Z^{\prime}=Z$ or $Z^{\prime \prime}=Z$.

\begin{cor}\label{Cor:BSlattice}
Assume the nondegeneracy condition~\eqref{eq:nondegcond}. Then:
\begin{enumerate}[label=\textnormal{(\roman*)}]
\item Buffering structures form a distributive lattice under the 
operations~\eqref{Eq:SubnetworkOps},~and we have
      \[
      \gamma_{J_1}\cup\gamma_{J_2}=\gamma_{J_1\cup J_2},
      \qquad
      \gamma_{J_1}\cap\gamma_{J_2}=\gamma_{J_1\cap J_2}.
      \]
      In particular buffering structures are closed under componentwise union and intersection,
      and their number equals the number of order ideals of $(\mathcal P,\preceq)$, which is at
      most $2^{p}$.
\item The nonempty join-irreducible buffering structures are exactly the subnetworks
      \[
      \gamma_{\downarrow M_a},\qquad \downarrow M_a:=\{M_c\in\mathcal P\mid M_c\preceq M_a\},
      \qquad a\in[p].
      \]
      Moreover, for every reaction $\rho\in M_a$ the subnetwork $\gamma_{\downarrow M_a}$ is the
      smallest buffering structure whose reaction set contains $\rho$. In particular, a buffering
      structure containing one reaction of a block contains the whole block.
\item The minimal nonempty buffering structures are the block subnetworks $v_a=(\mathbf M_a,M_a)$
      with $M_a$ minimal for $\preceq$, that is, those $M_a$ that are sources of
      $\mathcal D_{\CCF}(\bar L)$.
\end{enumerate}
\end{cor}

\begin{proof}
(i) By Prop.~\ref{defprop:orderideal}(i) the order ideals of $(\mathcal P,\preceq)$
form a distributive sublattice of~$2^{\mathcal P}$. Since the blocks $M_a$ are pairwise disjoint,
and likewise the sets $\mathbf M_a$, forming unions and intersections of order ideals corresponds
to forming componentwise unions and intersections of the associated subnetworks, that is
$\gamma_{J_1}\cup\gamma_{J_2}=\gamma_{J_1\cup J_2}$ and
$\gamma_{J_1}\cap\gamma_{J_2}=\gamma_{J_1\cap J_2}$. As $J\mapsto\gamma_J$ is an isomorphism of
partially ordered sets by Thm.~\ref{Thm:BSideals}, it is an isomorphism of lattices.

(ii) In the lattice of order ideals of a finite poset the nonempty join-irreducible elements are exactly
the principal ideals $\downarrow M_a$, showing the first part of (ii). Let $\rho$ be a reaction contained in a buffering structure $\gamma_J$, and let $M_a$ be the unique block containing $\rho$.
Since the blocks are pairwise disjoint, a reaction
$\rho\in M_a$ lies in the reaction set $\bigcup_{M_c\in J}M_c$ of $\gamma_J$ if and only if $M_a\in J$. Since $J$ is an order ideal this is equivalent to 
$\downarrow M_a\subseteq J$. By Thm.~\ref{Thm:BSideals} the smallest buffering structure whose
reaction set contains $\rho$ is thus $\gamma_{\downarrow M_a}$, and its reaction set contains all
of $M_a$.

(iii) The minimal nonempty order ideals are the singletons $\{M_a\}$ with $M_a$ minimal for
$\preceq$, and $\gamma_{\{M_a\}}=(\mathbf M_a,M_a)=v_a$. A block is minimal for $\preceq$ exactly
when it has no incoming edge in the Hasse diagram $\mathcal D_{\CCF}(\bar L)$, that is, when it is
a source.
\end{proof}

Part (i) of Cor.~\ref{Cor:BSlattice} is \cite[Cor.~1]{hirono2021structural} which is obtained here, without submodularity, directly from the identification of the buffering
structures with order ideals of $(\mathcal P,\preceq)$. The subnetwork
$\gamma_{\downarrow M_a}$ of part (ii) is the \emph{minimal influence block generated by a reaction} in
the sense of \cite[Eq.~(2.20)]{BF18}, and the family
$\{\gamma_{\downarrow M_1},\dots,\gamma_{\downarrow M_p}\}$ is the \emph{minimum set of buffering
structures} of \cite[Def.~4 and Thm.~6]{YHOM:24}. The observation that $\gamma_{\downarrow M_a}$ are exactly the
join-irreducible elements of the lattice, indexed by the blocks of the CCF, is what
Corollary~\ref{Cor:BSlattice} adds. 

\begin{example}\label{Ex:LatticeRunning}
We continue Example~\ref{Ex:IdealsRunning} and read off Cor.~\ref{Cor:BSlattice} for the
running example. Part~(i) was already illustrated in \eqref{Eq:RunningExIllustratePart1}.
For part~(ii), the three principal order ideals are $\downarrow M_1=\{M_1\}$,
$\downarrow M_2=\{M_1,M_2\}$ and $\downarrow M_3=\{M_3\}$, so the join-irreducible buffering
structures are
\[
\gamma_{\downarrow M_1}=\bigl(\{X_1\},\{1,2\}\bigr),\qquad
\gamma_{\downarrow M_2}=\bigl(\{X_1,X_2\},\{1,2,3,4\}\bigr),\qquad
\gamma_{\downarrow M_3}=\bigl(\{X_3\},\{5,6\}\bigr),
\]
The remaining nontrivial
buffering structure $\gamma_{\{M_1,M_3\}}=(\{X_1,X_3\},\{1,2,5,6\})$ is not join-irreducible: it
is the union of $\gamma_{\downarrow M_1}$ and $\gamma_{\downarrow M_3}$.

For part~(iii), the minimal blocks are $M_1$ and $M_3$, the two sources of
$\mathcal D_{\CCF}(\bar L)$, so the minimal nonempty buffering structures are
$v_1=(\{X_1\},\{1,2\})$ and $v_3=(\{X_3\},\{5,6\})$. The example shows that (ii) and (iii) are
genuinely different: $\gamma_{\downarrow M_2}$ is join-irreducible but not minimal, and the block
subnetwork $v_2=(\{X_2\},\{3,4\})$ is not a buffering structure at all, in accordance with
Example~\ref{Ex:SurplusRunning}, where we saw that $\{3,4\}$ is the reaction set of no buffering
structure.
\end{example}

We conclude this section by characterizing the reducibility of the symbolic Jacobian determinant in terms of buffering structure: the necessity part is a novel conclusion of this paper.

\begin{prop}\label{thm:q}
Assume the nondegeneracy condition~\eqref{eq:nondegcond}. Then the symbolic Jacobian determinant $\det(SR)$ is a reducible polynomial
in $\mathbb R[r_{\rho m}]$ if and only if the network admits a buffering structure
$(\mathbf M',\mathbf E')$ with $\emptyset\neq\mathbf M'\subsetneq\mathbf M$.
\end{prop}

\begin{proof}
Let $A$ be the LM-matrix associated to $SR$ and $\bar L$ its combinatorial canonical form, as fixed at the beginning of this subsection.
Since $\bar L=P_r A P_c$, the permutations change the determinant only by a sign, so
Prop.~\ref{Prop:CSGaledualTrick} and (CCF1) give
\[
\det (SR)=\delta'\,\det(\bar L)=\delta'\prod_{a=1}^{p}\det(\bar L[\tau(M_a),\tau(M_a)])
\qquad\text{for some }\delta'\in\mathbb R\setminus\{0\}.
\]
By~\eqref{Eq:BlockSpecies}, the species columns lying in the block
$\tau(M_a)$ are exactly those of $\mathbf M_a$. Hence, if $\mathbf M_a=\emptyset$, the block
carries no variable, so $|M_a|=1$ by Lemma~\ref{lem:symrow} and the corresponding factor is a
nonzero element of $\mathbb R$ by (CCF3). If $\mathbf M_a\neq\emptyset$, the block is a
nonsingular LM-irreducible LM-matrix with nonempty symbolic part, so its determinant is
irreducible by Thm.~\ref{thm:murotairr}. Since a product of one nonconstant irreducible
polynomial with units is irreducible, $\det(SR)$ is reducible if and only if at least two blocks
satisfy $\mathbf M_a\neq\emptyset$.

It remains to see that this happens exactly when the network has a buffering structure with
$\emptyset\neq\mathbf M'\subsetneq\mathbf M$. Suppose $\mathbf M_{b_1},\mathbf M_{b_2}\neq\emptyset$
with $b_1\neq b_2$. If $M_{b_2}\not\preceq M_{b_1}$, put $J:=\downarrow M_{b_1}$; otherwise
$M_{b_2}\prec M_{b_1}$ and we put $J:=\downarrow M_{b_2}$, which does not contain $M_{b_1}$. In
either case $J$ is an order ideal containing exactly one of $M_{b_1},M_{b_2}$, so by
Thm.~\ref{Thm:BSideals} the species set $\bigcup_{M_a\in J}\mathbf M_a$ of $\gamma_J$ is nonempty
and misses the species of the other block. Conversely, if $\mathbf M_a\neq\emptyset$ for exactly
one index $a$, then $\bigcup_{M_c\in J}\mathbf M_c$ is either empty or all of $\mathbf M$ for every
order ideal $J$, so by Thm.~\ref{Thm:BSideals} no buffering structure has
$\emptyset\neq\mathbf M'\subsetneq\mathbf M$.
\end{proof}

\begin{example}
To see that the condition $\mathbf{M}'\subsetneq \mathbf{M}$ is needed, consider the following example with one species and two reactions:
$
X_1 \longrightarrow 2X_1, \; 0 \longrightarrow X_1.
$
The stoichiometric and symbolic reactivity matrices are, respectively, $S = (1,1)$, $R = (r_{11} ,0)^T$ and \[A = \begin{pmatrix}
    r_{11} & 1 \\ 0 & -1
\end{pmatrix},\]
with irreducible determinant $\det(A)=-r_{11}$. The subnetwork
$(\{X_1\},\{1\})$ is output-complete and satisfies $\lambda(\{X_1\},\{1\})=1-1-0=0$, so it is a
nontrivial buffering structure, but its species set is all of $\mathbf M$, so the condition
$\emptyset\neq\mathbf M'\subsetneq\mathbf M$ of Prop.~\ref{thm:q} fails.
\end{example}

\subsection{Sensitivity analysis}\label{sec:sensitivity}

Sensitivity analysis studies the response of the system to external perturbations. For simplicity of presentation, we consider linear perturbations: let $\boldsymbol{\alpha} \in \mathbb{R}^{|\mathbf{E}|}_{\ge0}$ be a nonnegative \emph{perturbation vector}, and let $\mathbf{r}(\mathbf{x})$ be a kinetic model as in Def.~\ref{def:kineticmodel}. The \emph{perturbed kinetic model} is defined as
\begin{equation*}\label{eq:reactionperturbation}
\mathbf{r}^{\varepsilon}(\mathbf{x}):=(1+\varepsilon \boldsymbol{\alpha})  \ast \mathbf{r}(\mathbf{x}), 
\end{equation*}
where $\ast$ denotes the coordinatewise product of two vectors.

Let $\mathbf{\bar x} \in \mathbb{R}_{>0}^{\vert \mathbf{M} \vert}$ be a \emph{nondegenerate}
equilibrium of the ODE system associated with the kinetic model $\mathbf{r}(\mathbf{x})$ as in~\eqref{eq:main}. Here, nondegenerate means that the Jacobian matrix $S\Jac_{\mathbf{r}}(\mathbf{\bar x})$ is nonsingular.
We address how the equilibrium concentration vector $\mathbf{\bar x}$ responds to perturbations of the reaction rate function.
Consider the perturbed equilibrium equation
\begin{equation}\label{eq:perturbedss}
0=f^\varepsilon(\mathbf{x})
=S\mathbf{r}^\varepsilon(\mathbf{x}).
\end{equation}
By the Implicit Function Theorem, there exists a differentiable function $\mathbf{\bar x}(\varepsilon) \colon (-\varepsilon_*,\varepsilon_* ) \to\mathbb{R}_{>0}^{\vert \mathbf{M} \vert} $ such that $f^{\varepsilon}( \mathbf{\bar x}(\varepsilon)) = 0$ for all $\varepsilon \in (-\varepsilon_*,\varepsilon_* ) $ and $\mathbf{\bar x}(0) = \mathbf{\bar x}$.
We call 
\begin{equation}\label{eq:concresp}
    \delta \mathbf{\bar x}^{\boldsymbol{\alpha}}\coloneqq\dfrac{\partial\bar{\mathbf{x}}(\varepsilon)}{\partial \varepsilon}\bigg|_{\varepsilon=0\;}
\end{equation}
the \emph{concentration response} to an $\boldsymbol{\alpha}$-perturbation.

The concentration response vector can be written in terms of $\boldsymbol{\alpha}$ and the inverse of the $A$ matrix from~\eqref{Eq:AMatrix} as follows. Differentiating the identically zero function $f^{\varepsilon}( \mathbf{\bar x}(\varepsilon)) $ gives
\begin{equation}\label{eq:diffss}
    0=\dfrac{\partial f^\varepsilon(\mathbf{\bar x}(\varepsilon))}{\partial \varepsilon}=S \dfrac{\partial \mathbf{r^\varepsilon(\bar{\mathbf{x}}(\varepsilon))}}  {\partial \varepsilon} =
    S \Big( \boldsymbol{\alpha } \ast \mathbf{r}(\mathbf{\bar x}(\varepsilon)) + (1+\varepsilon\boldsymbol{\alpha}) \ast \big( \Jac_\mathbf{r}(\mathbf{\bar x}(\varepsilon)) \cdot \tfrac{\partial \mathbf{\bar x}(\varepsilon)}{\partial \varepsilon}\big)\Big).
\end{equation}
In the symbolic spirit of this paper, we replace $\Jac_\mathbf{r}(\mathbf{\bar x}(\varepsilon))$ by the symbolic reactivity matrix $R$ as in~\eqref{Eq:SymbReactivity}.
Evaluating~\eqref{eq:diffss} at $\varepsilon=0$ gives that $\boldsymbol{\alpha} \ast \mathbf{r}(\mathbf{\bar x}(0)) + R\delta \mathbf{\bar x}^{\boldsymbol{\alpha}}$ lies in the kernel of $S$. 
Thus for any choice $\{c^1,...,c^K\}$ of a basis of $\ker(S)$ there exist $\mu_1^{\boldsymbol{\alpha}}, \dots ,\mu_K^{\boldsymbol{\alpha}} \in \mathbb{R}(r_{\rho m})$
such that
\begin{equation*}\label{eq:introA}
      \boldsymbol{\alpha}\ast \mathbf{r}(\mathbf{\bar x}(0)) + R \cdot \delta \mathbf{\bar x}^{\boldsymbol{\alpha}}= - \sum_{i=1}^K \mu_i^{\boldsymbol{\alpha}} c^i ,
\end{equation*}
where the minus convention is due to consistency of notation with the $A$ matrix from \eqref{Eq:AMatrix}. 
If $A = [\,R\, \vert\, C\,]$ and the columns of $C$ are given by the basis vectors $c^1,\dots,c^K$, we have
\begin{equation*}
    A \begin{pmatrix}
        \delta \mathbf{\bar x}^{\boldsymbol{\alpha}} \\ \mu^{\boldsymbol{\alpha}}    \end{pmatrix}=-\boldsymbol{\alpha} \ast \mathbf{r}(\mathbf{\bar x}(0)) 
\end{equation*}
and therefore the concentration response is structurally encoded in the inverse of the $A$ matrix:
\begin{equation}\label{eq:responsesformula}
    \begin{pmatrix}
        \delta \mathbf{\bar x}^{\boldsymbol{\alpha}}\\ \mu^{\boldsymbol{\alpha}}    \end{pmatrix}=-\;A^{-1}\,(\boldsymbol{\alpha} \ast \mathbf{r}(\mathbf{\bar x}(0)) ).
\end{equation}

\subsection{Targeted perturbations and influence-relations} \label{sec:targeted}
The linearity of the responses in~\eqref{eq:responsesformula} suggests to consider \emph{targeted perturbations} of the form $\boldsymbol{\alpha}=\mathbf{e}_{\rho}$ for a uniquely perturbed reaction~$\rho$, since the responses to such perturbations span the responses to any general vectorial perturbation~$\boldsymbol{\alpha}$.
Moreover, after normalization, the concentration response vector of a targeted perturbation of reaction $\rho$ can be read off in the first $\mathbf{M}$ entries in the $\rho^{th}$ column of $-A^{-1}$, by~\eqref{eq:responsesformula}.
Consequently, we say that a \emph{reaction $\rho$ influences a species $X_m$} if  for $\boldsymbol{\alpha} = \mathbf{e}_\rho$ we have $(\delta \mathbf{\bar x}^{\boldsymbol{\alpha}})_m \neq 0$, or equivalently if $(A^{-1})_{m \rho} \neq 0$. Here, again, $(\delta \mathbf{\bar x}^{\boldsymbol{\alpha}})_m \neq 0$, $(A_{m\rho})^{-1}\neq 0$ mean nonzero as a symbolic rational expressions in $\mathbb{R}(r_{\rho m})$. Thus, such  structural influence need not imply an actual nonzero response at every specialization of the variables $r_{\rho m}$.

The goal of this section is to introduce a graph, derived from the combinatorial canonical form
of $A$, that completely describes the influence relations between reactions and species. We keep
the notation of Sec.~\ref{sec:bs-lattice}: $\bar L=P_r A P_c$ is the fixed combinatorial
canonical form of $A$, the permutations $\tau,\sigma$ are those fixed in Sec.~\ref{sec:bs-lattice}, the row/column
partition of $A$ is $\mathbf E=M_1\sqcup\dots\sqcup M_p$ with associated species sets
$\mathbf M_1,\dots,\mathbf M_p$ as in~\eqref{Eq:BlockSpecies}, and $v_a=(\mathbf M_a,M_a)$ are
the block subnetworks of~\eqref{Eq:BlockData}. Recall that the row/column partition of $\bar L$
is $\tau(M_1),\dots,\tau(M_p)$, that $M_a\preceq M_b$ if and only if
$\tau(M_a)\preceq\tau(M_b)$, and that $\mathbf M_a=\{X_m\in\mathbf M\mid\sigma(m)\in\tau(M_a)\}$
by~\eqref{Eq:BlockSpecies}. From
\[
\bar L^{-1} = P_c^{-1} A^{-1} P_r^{-1}, \qquad\text{or equivalently from }\qquad A^{-1} = P_c \bar L^{-1} P_r ,
\]
we obtain $(A^{-1})_{m \rho} = (\bar L^{-1})_{\sigma(m)\,\tau(\rho)}$, so an entry of $A^{-1}$ is
nonzero precisely when the corresponding entry of $\bar L^{-1}$ is.

We define the \emph{influence graph} $\mathcal G$ of $\mathcal N$ as the directed graph with
vertex set $\{v_1,\dots,v_p\}$, where there is a directed edge from $v_b$ to $v_a$ if and only if
there is a directed edge from $M_a$ to $M_b$ in the CCF directed graph
$\mathcal D_{\CCF}(\bar L)$. Thus the edges of $\mathcal G$ carry the orientation opposite to
those of $\mathcal D_{\CCF}(\bar L)$, which aligns our definition with the reaction network
literature.

\begin{theorem}[Theorem \ref{theorem:B}]\label{thm:influencegraph}
Assume that the symbolic Jacobian determinant $\det(SR)$ of a reaction network $\mathcal{N} = (\mathbf{M},\mathbf{E})$ is not zero, and let $\mathcal G$ be its influence graph.
A reaction $\rho \in \mathbf{E}$ influences a species $X_m \in \mathbf{M}$ if and only if there is a directed path from $v_b$ to $v_a$ in the influence graph, where $a,b\in[p]$  are the indices with $X_m\in\mathbf M_a$ and $\rho\in M_b$.
\end{theorem}

\begin{proof}
Let $a,b\in[p]$ be the indices such that $X_m \in \mathbf M_a$ and $\rho \in M_b.$ This means that $\sigma(m)\in\tau(M_a)$, and $\tau(\rho)\in\tau(M_b)$. 
The reaction $\rho$ influences $m$ if and only if 
\[(A^{-1})_{m \rho} = (\bar L^{-1})_{\sigma(m) \tau(\rho)} \neq   0.\]
By Thm.~\ref{thm:ccfinv}, this is equivalent to $\tau(M_a)\preceq\tau(M_b)$, hence to $M_a \preceq M_b$, which in turn is equivalent to the existence of a directed path from $v_b$ to $v_a$ in $\mathcal G$.
\end{proof}

Thm.~\ref{thm:influencegraph} provides a complete characterization of which reactions influence which species concentrations. We note that this information (that is, the influence graph) can be computed symbolically without numerical error in polynomial time using the algorithm of~\cite{murota_algo}. We refer to Sec.~\ref{Sec:Computation} for more details on computational aspects.

Results related to Thm.~\ref{thm:influencegraph} have appeared in the literature.
Brehm and Fiedler \cite{BF18} gave exact formulas for the sensitivity responses  
and established `transitivity of influence' via the mathematical argument of Lemma~\ref{lem:transition}. Transitivity guarantees that if a reaction $\rho_1$ influences a species $X_j$, which is a reactant of a reaction $\rho_2$ that influences a species $X_k$, then $\rho_1$ influences $X_k$. A posteriori (after numerically computing $A^{-1}$ via probabilistic method of polynomial identity testing \cite{schwartz:1980}), they construct an influence graph via the preorder induced by transitivity.
The authors in \cite{BF18} regard their influence graph as related indeed to Murota's theory, although they do not formally pursue this connection.

Yamauchi, Hishida, Okada, and Mochizuki \cite{YHOM:24} subsequently obtained a closely related hierarchy directly from buffering structures. They associate to each reaction $\rho$ the minimal buffering structure containing $\rho$ and organize these structures by inclusion, obtaining a `hierarchy graph'. The construction crucially relies on the transitivity theorem \cite{BF18}, and they obtain the hierarchy graph numerically, which also requires computing $A^{-1}$.
The authors further prove that this structural hierarchy coincides with a hierarchy of nonzero responses: influence between the resulting disjoint hierarchy components is `all or none', so that every parameter in an upstream component influences every species in each downstream component. Thus, reachability in the hierarchy graph gives a genuine characterization of metabolite influence.

Huang, Okada, and Mochizuki \cite{HOM:25} also introduced another `influence graph' based on the lattice of buffering structures. Their results show that nonzero influence can occur only if two elements are connected by reachability in their influence graph, although the sufficiency of this construction, even for metabolite responses, is not explicitly stated there. The approaches of \cite{HOM:25} and of  \cite{YHOM:24} appear closely related, though their equivalence has not been formally established.  In particular, both approaches partially extend to conserved quantities; however, the all-to-all hierarchy statement in \cite{YHOM:24} relies on \cite{BF18}, which assumes no conserved quantities.

\section{Computations}
\label{Sec:Computation}
In this section, we briefly present our proof-of-concept \texttt{Julia} implementation for computing the combinatorial canonical form of an LM-matrix and the influence graph of a reaction network. We also discuss how the resulting combinatorial information can be used to compute the symbolic Jacobian determinant. Our implementation of the CCF algorithm is based on the algorithm presented in~\cite{murota_algo} which runs in polynomial time. The code is available at
\begin{center}
\url{https://github.com/ArneKuhrs/LayeredMixedMatrices.jl}\,.
\end{center}
To showcase how to use our code, we consider the \emph{nominal cell model} from~\cite{Schliemann2011}. This network has $47$ species and $106$ reactions. The list of all reactions can be found in  \texttt{ODEbase}~\cite{LuedersSturmRadulescu:22}. Using the package \texttt{Catalyst.jl}~\cite{Catalyst}, one can easily enter the reactions~as
\begin{alltt}
rn = @reaction_network begin 
    k1, TNFR --> TNFR_E
    k2, 0 --> TNFR
    \(\vdots\)
    k87, PARP --> cPARP
    (k88,l88), BAR + Casp8 <--> BAR_Casp8
end
\end{alltt}
After computing the reactivity matrix \texttt{R} and a Gale dual \texttt{C} of the stoichiometric matrix,
the function \texttt{ccf(R,C)} returns the combinatorial canonical form. It is worth noting that the computation takes only $0.0008684$ seconds\footnote{The computations were performed on a MacBook Air with an M4 chip and 16 GB RAM. The time is an
average of five runs.} for the nominal cell model.
In this example, the CCF has $50$ irreducible blocks: $41$ of size one, $3$ of size two, $2$ of size five, and one each of sizes $3, 7$, $13$, and $26$. The CCF-directed graph and equivalently the influence graph, which can be computed via the function \texttt{block\_poset(ccf\_res)}, has $39$ edges between the $50$ vertices.

Due to the size of the network, a direct computation of the symbolic Jacobian determinant is infeasible. We discuss how the combinatorial canonical form can be used to facilitate this computation.
Since the CCF consists of 50 irreducible blocks, the Jacobian determinant can have at most $50$ irreducible factors. In fact, only $26$ of these blocks contain symbolic variables, giving $26$ factors of the determinant. The following table summarizes the sizes of these blocks (omitting the ones of size $1$), and the size of the corresponding factors of the determinant.
\[
\begin{array}{c|ccccccccc}
\text{size of the block}
& 2 & 2 & 2 & 3 & 5 & 5 & 7 & 13 & 26 \\
\hline
\hline
\text{number of variables}
& 2 & 2 & 2 & 4 & 5 & 6 & 8 & 15 & 36 \\
\hline
\text{number of monomials}
& 2 & 2 & 2 & 3 & 4 & 4 & 11 & 114 & 57341
\end{array}
\]
Thus, for the nominal cell model, the Jacobian determinant is a polynomial with $27{.}611{.}755{.}776$ monomials, which can nevertheless be computed (in factored form) using the CCF.

\section{Discussion}\label{sec:discussion}

In this paper, we investigated layered mixed matrices from both an algebraic and a reaction network perspective, with the main goal of exposing direct connections between
the two.

On the algebraic side, we have related LM-matrices to linearly transformed sparse generic matrices $ST$, a class of column-graded matrices naturally connected to LM-matrices through Gale duality. This connection yields a combinatorial characterization of the irreducibility of $\det(ST)$. 
We have also contributed to the structural theory of LM-matrices by characterizing in Thm.~\ref{thm:ccfinv} the nonzero pattern of the inverse of an LM-matrix in CCF entirely in terms of the associated block partial order.

These algebraic results serve naturally reaction network theory, where matrices of the form $ST$ arise as Jacobian matrices.
Thus, the factorization of the Jacobian determinant can  be characterized in terms of the associated LM-matrix. We have further connected LM-matrix theory to sensitivity analysis, recovering the lattice of buffering structures as the lattice of order
ideals of the CCF block poset. In this way, our results provide reaction network theory with a new language and a broad set of tools from LM-matrix theory, while reaction networks provide LM-matrix theory with a relevant application that naturally poses new mathematical questions.

A central example of this interplay is again Thm.~\ref{thm:ccfinv}, which characterizes the nonzero entries of the inverse of an LM-matrix in CCF. Such a result does not appear in the original LM-matrix theory, where the emphasis was on properties of arbitrary LM-matrices rather than specifically on their CCF. For reaction networks, however, the matrix of sensitivities of species concentrations to reaction-rate perturbations can be identified with the inverse of a matrix in CCF, making this question directly relevant. Thm.~\ref{thm:influencegraph}, which is simply the reaction network interpretation of Thm.~\ref{thm:ccfinv}, yields therefore a characterization of the structurally nonzero sensitivity responses. 

As a first case of interest, we have focused here on invertible LM-matrices and, consequently, on reaction networks whose Jacobian
determinant is not identically zero. This excludes networks with linear conserved quantities, which are however common in applications. Since
sensitivity results are already available in this more general setting, an important direction for future work is to understand how the tools
of LM-matrix theory can be extended to reaction networks with conserved quantities.\\

\textbf{Acknowledgments.} We thank Atsushi Mochizuki and Takashi Okada for their generous support and encouragement, and in particular  for clarifying details in the literature on buffering structures. 
M.L.T. and A.K. are grateful to Bernd Sturmfels for bringing the literature on sparse generic matrices to their attention.
N.V. is indebted to Bernold Fiedler and Bernhard Brehm for opening the gates to sensitivity theory. M.L.T. is supported by the project 
101183111-DSYREKI-HORIZON-MSCA2023-SE-01 
"Dynamical Systems and Reaction Kinetics Networks"; N.V. is supported by the Novo Nordisk
Foundation (grant NNF21OC0066551 `MATOMIC').

\bibliographystyle{amsalpha}
\bibliography{biblio}{}

\end{document}